\documentclass{amsart}
\usepackage{biblatex} 
\usepackage{mathtools} 
\usepackage{xifthen} 
\usepackage{hyperref} 
\usepackage{derivative} 
\usepackage{graphicx} 

\hypersetup{
	colorlinks,
	linkcolor={red},
	citecolor={magenta},
	urlcolor={blue}
}

\theoremstyle{plain}
\newtheorem{theorem}{Theorem}

\newtheorem{proposition}{Proposition}

\newtheorem{corollary}{Corollary}

\theoremstyle{definition}
\newtheorem{definition}{Definition}[section]

\theoremstyle{remark}
\newtheorem{example}{Example}

\providecommand{\set}[2][]{
	\ifthenelse{\isempty{#1}}{
		\left\{#2\right\}
	}{
		\left\{\,#1\;\middle|\;#2\,\right\}}
}
\providecommand{\abs}[1]{\left\lvert#1\right\rvert} 
\providecommand{\norm}[1]{\left\lVert#1\right\rVert} 
\providecommand{\A}{\mathbb{A}} 
\providecommand{\C}{\mathbb{C}} 
\providecommand{\J}{\mathbb{J}} 
\providecommand{\N}{\mathbb{N}} 
\providecommand{\R}{\mathbb{R}} 
\providecommand{\Z}{\mathbb{Z}} 

\title{A rock-paper-scissors Mandelbrot set}

\author[C.~Aten]{Charlotte~Aten}
\urladdr{\href{https://aten.cool}{https://aten.cool}}
\email{\href{mailto:charlotte.aten@posteo.net}{charlotte.aten@posteo.net}}


\begin{document}

\begin{abstract}
The titular object of this paper is an analogue of the Mandelbrot set over the \(3\)-dimensional real algebra whose multiplication is the bilinear extension of the rock-paper-scissors operation. In order to study the dynamics of the mappings \(x\mapsto x^2+c\) in this setting, a notion of holomorphy for functions on general finite-dimensional real algebras is introduced. Under a mild assumption it is shown that for such algebras holomorphy is always equivalent to solving a finite system of linear first-order PDEs generalizing the Cauchy-Riemann equations. Code is provided for generating animations of these fractals and for exploring them in a video game format.
\end{abstract}

\maketitle

\begin{center}
	\includegraphics[height=7cm]{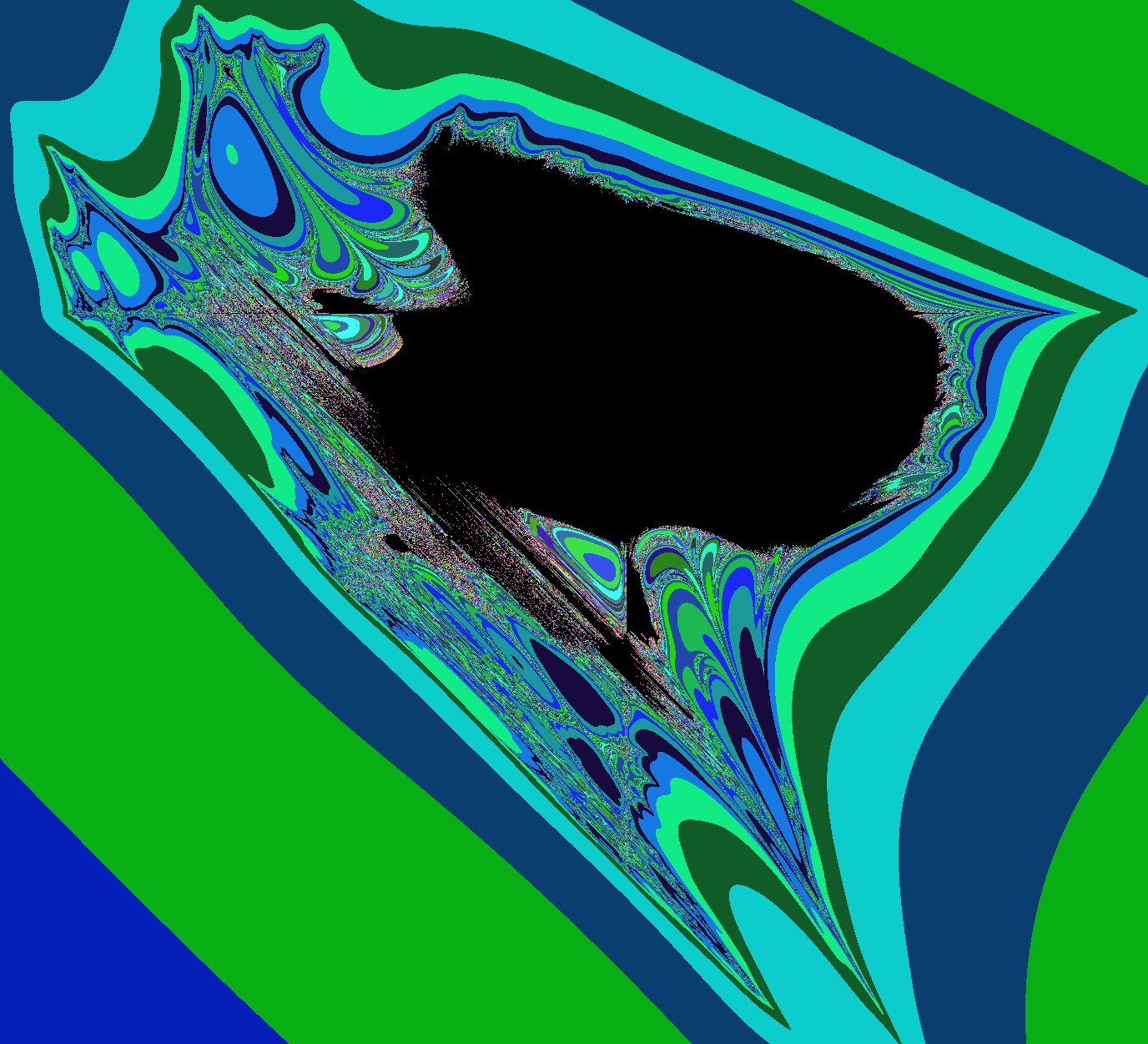}
\end{center}

\tableofcontents

\section{Introduction}
\label{section:introduction}
There has been interest in the past in a higher-dimensional (and in particular \(3\)-dimensional) analogue of the Mandelbrot set\cite{barrallo2010}. The study of the Mandelbrot set belongs to the discipline of complex dynamics. Higher dimensional (and in particular \(3\)-dimensional) analogues of the complex numbers and their holomorphic functions have been sought since the 19th century. The \(4\)-dimensional quaternions, by virtue of being a division algebra, have received the most attention in this regard. Hamilton himself observed that it seemed difficult to endow the quaternions with an analogue of differentiation which makes the function \(x\mapsto x^2\) holomorphic\cite[p.998]{deavours1973}.

In \autoref{section:holomorphic_functions} we define holomorphy for a function of a single variable over a finite-dimensional real algebra. Our definition implies that for a holomorphic function \(f\) we have
	\[
		vf'(x)=f'(x)v=\nabla_vf(x)
	\]
for all directions \(v\) and all \(x\), although it is stronger than this condition. The demand that \(vf'(x)=f'(x)v\) is already quite strong, for this means that any holomorphic function must have the image of its derivative contained entirely in the center of the underlying algebra. For the quaternions, this would mean that a holomorphic function must have a purely real-valued derivative.

While one might imagine that we would end up with a paucity of examples of such functions, it turns out that the existence of several classes of holomorphic functions may be demonstrated for large families of finite-dimensional real algebras. What's more, the holomorphy of these classes is in each case equivalent to an algebraic property of the underlying real algebra. For example, all constant functions on a finite-dimensional real algebra \(\A\) are holomorphic if and only if no nonzero element of \(\A\) annihilates \(\A\). Under the relatively weak assumption of a single divisible element (that is, an element with either its left- or right-multiplication map invertible), we can provide similar results not only for classes of functions like \(x\mapsto ax\) for \(a\in\A\) but also for rules of differentiation. Under this assumption, the holomorphic functions on \(\A\) form a real vector space in the usual manner. They do not in general form an algebra, but the Product Rule is equivalent to every central element being middle nuclear and communal, a purely algebraic condition on \(\A\).

We find that in fact the unital finite-dimensional real algebras for which the Product Rule holds are precisely those which are associative and commutative. This doesn't completely resolve the question of which algebras \(\A\) carry an algebra of holomorphic functions under pointwise multiplication, but it answers the question when the derivative of a product is given by the familiar rule.

Also under the assumption of a divisible element, we find that the algebraic condition characterizing the Chain Rule is actually a weak form of the ``middle nuclear and communal'' condition for the Product Rule. Thus, that the Product Rule holds implies that the Chain Rule does as well, in this context.

We state one result as theorem, which is the existence of an analogue of the Cauchy-Riemann equations. Our \autoref{theorem:Cauchy-Riemann_equations} says that for any \(n\)-dimensional real algebra containing a divisible element there exists a system of \(n^2\) linear first-order partial differential equations whose satisfaction is equivalent to holomorphy. This is done by giving an example of such a system. The usual Cauchy-Riemann equations are a special case of this result.

In \autoref{section:holomorphic_functions_on_the_RPSions} we study the holomorphic functions on the RPSions, a \(3\)-dimensional real algebra \(\J\) whose multiplication is the rock-paper-scissors operation on its standard basis. For a discussion of the rock-paper-scissors operation in a more combinatorial context, see the author's previous work on the subject\cite{aten2020}. We show that one system of Cauchy-Riemann-like equations for \(\J\) is:
	\begin{align*}
		&\pdv{u}{y}=\pdv{v}{z}=\pdv{w}{x}=0 \\
		&2\pdv{u}{z}=\pdv{u}{x}+\pdv{v}{y}-\pdv{w}{z} \\
		&2\pdv{v}{x}=-\pdv{u}{x}+\pdv{v}{y}+\pdv{w}{z} \\
		&2\pdv{w}{y}=\pdv{u}{x}-\pdv{v}{y}+\pdv{w}{z}
	\end{align*}
We then use these equations to show that the mapping \(x\mapsto x^2+c\) is holomorphic over \(\J\).

In \autoref{section:dynamics_in_real_algebras} we generalize the setup for complex dynamics to our notion of holomorphic function. We frequently find that \(f\colon\A\to\A\) is holomorphic but its iterates \(f^n\colon\A\to\A\) are not, since the Chain Rule does not hold. This is certainly the case in \(\J\). Since we would like to define attracting and repelling periodic points, this is a difficulty. It turns out we can substitute \(\abs{(f^n)'(x)}\) with the operator norm of a certain linear transformation associated to the point \(x\) and the function \(f\). This definition collapses to the usual one in the case of \(\C\).

The software associated to this work is discussed in \autoref{section:software_for_fractal_animations} and \autoref{section:video_game_fractal_visualization}. These sections detail the mathematical choices made in visualizing fractals associated to these hypercomplex dynamical systems on the computer. In \autoref{section:future_directions} we conclude with several possible directions for future research.

We recall some terminology to be used throughout. Given a function \(f\colon\R^m\to\R^n\) and \(v\in\R^m\), the \emph{directional derivative} \(\nabla_vf\colon\R^m\to\R^n\) is given by
	\[
		\nabla_vf(x)=\lim_{\substack{h\to0\\h\in\R}}\frac{f(x+hv)-f(x)}{h}.
	\]
The \emph{Jacobian} of a differentiable function \(f\colon\R^m\to\R^n\) at a point \(x\in\R^m\) is the linear transformation \(J_f(x)\colon\R^m\to\R^n\) such that \(J_f(x)v=\nabla_vf(x)\) for each \(x\in\R^m\). Recall that we have the Chain Rule for differentiable functions \(f\colon\R^m\to\R^n\) and \(g\colon\R^n\to\R^k\), which says that
	\[
		J_{g\circ f}(x)=J_g(f(x))\circ J_f(x).
	\]

This work lies at the intersection of two different meanings of the word ``algebra''. What we refer to here as a ``real algebra'' is a real vector space equipped with a bilinear multiplication, which is not assumed to be unital, commutative, or associative. These structures are sometimes called ``nonassociative algebras''\cite{schafer1961}. The magmas we discuss are algebras in the sense of universal algebra. See \cite{smith1999} or \cite{bergman2012} for a more elementary introduction to universal algebra and \cite{bergman2015} for a graduate-level treatment.

Given any finite magma \(\mathbf{A}=(A,f\colon A^2\to A)\) we can consider the finite-dimensional real algebra \(\R[\mathbf{A}]\) whose elements are \(\R\)-linear combinations of the members of \(A\) and whose multiplication is the bilinear extension of \(f\) to the set \(\R[A]\) of such linear combinations. That is, we define multiplication in \(\R[\mathbf{A}]\) by
	\[
		\left(\sum_{x\in A}u_xx\right)\left(\sum_{y\in A}v_yy\right)=\sum_{\mathclap{(x,y)\in A^2}}u_xv_yf(x,y).
	\]

The case of \(\R[\mathbf{G}]\) for a finite group \(\mathbf{G}\) is known as the \emph{group algebra} of \(\mathbf{G}\). These are well-studied in representation theory\cite{serre1977,fulton2004}. The monoid algebra \(\R[\mathbf{M}]\) for a finite monoid \(\mathbf{M}\) is similarly studied in the representation theory of finite monoids\cite{steinberg2016}. Both of these families of examples have an associative multiplication and a multiplicative identity, although their multiplication may not be commutative and in general division doesn't make sense.

\begin{definition}[Magma algebra]
\label{definition:magma_algebra}
By analogy with the group and monoid cases, we refer to a real algebra of the form \(\R[\mathbf{A}]\) as the \emph{magma algebra} of \(\mathbf{A}\).
\end{definition}

Sometimes the magma algebra \(\R[\mathbf{A}]\) is known as a \emph{convolution algebra}\cite[p.5]{wodzicki2015}.

Here we will continue to assume that \(\mathbf{A}\) is finite, although similar considerations can be made for infinite magmas by taking \(\R[A]\) to consist of finite linear combinations of members of \(A\) when \(A\) is infinite.

Other than the special case of the ring \(\R\) itself, which is \(\R[\mathbf{A}]\) when \(\mathbf{A}\) is the trivial magma, the Cayley-Dickson algebras are not of the form \(\R[\mathbf{A}]\) for some magma \(\mathbf{A}\). Unlike the magma algebras, the other Cayley-Dickson algebras do not admit a basis with respect to which the multiplication is a magma operation.

\begin{proposition}
\label{proposition:Cayley-Dickson_not_magma_algebra}
Suppose that \(\A\) is a finite-dimensional real algebra which contains a nonzero identity element \(e\) and an element \(i\) such that \(i^2=-e\). We cannot have \(\A=\R[\mathbf{A}]\) for any magma \(\mathbf{A}=(A,f)\).
\end{proposition}

\begin{proof}
Suppose towards a contraction that \(\A\) is a magma algebra. Let \(u=\sum_{a\in A}a\). We have that \(eu=u\), so taking the \(a\) component of both sides yields
	\[
		\sum_{\mathclap{\substack{(c,d)\in A^2\\f(c,d)=a}}}e_c=1.
	\]
Summing both sides over \(A\) givens
	\[
		\sum_{\mathclap{(c,d)\in A^2}}e_c=\abs{A}
	\]
so
	\[
		\abs{A}\sum_{c\in A}e_c=\abs{A}
	\]
and hence
	\[
		\sum_{c\in A}e_c=1.
	\]

On the other hand, taking the \(a\) component of \(i^2=-e\) yields
	\[
		\sum_{\mathclap{\substack{(c,d)\in A^2\\f(c,d)=a}}}i_ci_d=-e_a.
	\]
Summing both sides over \(A\) gives
	\[
		\sum_{\mathclap{(c,d)\in A^2}}i_ci_d=-\sum_{a\in A}e_a
	\]
so
	\[
		\left(\sum_{c\in A}i_c\right)^2=-\sum_{a\in A}e_a
	\]
and thus
	\[
		\sum_{a\in A}e_a\le0.
	\]
Since we already know that \(\sum_{a\in A}e_a=1\) we find that \(1\le0\), a contradiction.
\end{proof}

This paper is motivated by the study of the following magma algebra.

\begin{figure}
	\begin{center}
		\begin{tabular}{r|ccc}
			& \(r\) & \(p\) & \(s\) \\ \hline
			\(r\) & \(r\) & \(p\) & \(r\) \\
			\(p\) & \(p\) & \(p\) & \(s\) \\
			\(s\) & \(r\) & \(s\) & \(s\)
		\end{tabular}
	\end{center}
	\caption{Rock-paper-scissors magma}
	\label{figure:rock-paper-scissors_magma}
\end{figure}

\begin{definition}[RPSions]
\label{definition:rpsions}
The \emph{RPSions} are the magma algebra \(\J=\R[\mathbf{A}]\) where \(\mathbf{A}\) is the rock-paper-scissors magma depicted in \autoref{figure:rock-paper-scissors_magma}.
\end{definition}

The symbol \(\J\) is meant as a reference to the Japanese name for the game, which is transliterated as ``jan-ken-pon''. The RPSions are a nonassociative, commutative, \(3\)-dimensional real algebra without a multiplicative identity.

\section{Holomorphic functions}
\label{section:holomorphic_functions}
We would like a notion of holomorphic function which extends the idea of a complex-differentiable function to an arbitrary finite-dimensional real algebra. The following definitions are relevant.

\begin{definition}[Algebraic directional derivative]
\label{definition:algebraic_directional_derivative}
Let \(\A\) be a finite-dimensional real algebra, let \(f\colon\A\to\A\) be a function, and fix a point \(x\in\A\) and a direction \(v\in\A\). We say that \(\delta\in\A\) is a \emph{left algebraic directional derivative} of \(f\) at \(x\) in the direction \(v\) when
	\[
		v\delta=\nabla_vf(x)
	\]
and that \(\delta\in\A\) is a \emph{right algebraic directional derivative} of \(f\) at \(x\) in the direction of \(v\) when
	\[
		\delta v=\nabla_vf(x).
	\]
We say that \(\delta\in\A\) is an \emph{algebraic directional derivative} of \(f\) at \(x\) in the direction of \(v\) when \(\delta\) is both a left and right algebraic directional derivative of \(f\) at \(x\) in the direction \(v\).
\end{definition}

The chirality in the preceding definition is motivated as follows. If we were able to divide
	\[
		v\delta=\nabla_vf(x)
	\]
on the left by \(v\) we would have an expression much like
	\[
		\delta=\lim_{h\to0}\frac{f(x+hv)-f(x)}{hv},
	\]
which is similar to the definition of the derivative of a complex function \(f\colon\C\to\C\). In particular, this looks like the expression appearing in the production of the Cauchy-Riemann equations, where \(v\) is taken to be \(1\) or \(i\). Of course we don't need to worry about the distinction between left and right division in a field like \(\C\). Even in an algebra like \(\J\), which is commutative but is not a division algebra, the above definitions collapse to a single notion of ``algebraic directional derivative''.

Note that we did not exclude the possibility that \(v=0\) in the above definition. Any \(\delta\in\A\) is a directional derivative of \(f\) at \(x\) in the direction \(0\) because we always have that \(0\delta=\delta 0=0\) and
	\[
		\nabla_0f(x)=\lim_{h\to0}\frac{f(x+h(0))-f(x)}{h}=0.
	\]

\begin{definition}[Holomorphic function]
\label{definition:holomorphic_function}
Given a finite-dimensional real algebra \(\A\) and a function \(f\colon\A\to\A\), we say that \(f\) is \emph{holomorphic} at a point \(x\in\A\) when
	\begin{enumerate}
		\item \(f\) is differentiable at \(x\) and
		\item there exists a unique \(\delta\in\A\) such that for all directions \(v\in\A\) we have that \(\delta\) is an algebraic directional derivative of \(f\) at \(x\) in the direction \(v\).
	\end{enumerate}
When \(f\) is holomorphic at \(x\) we denote this unique value \(\delta\) by \(f'(x)\).
\end{definition}

As usual in complex analysis, we will say that \(f\colon\A\to\A\) is holomorphic on an open set \(\Omega\subset\A\) when \(f\) is holomorphic at each point in \(\Omega\) and we will say that \(f\colon\A\to\A\) is holomorphic (or \emph{entire}) when \(f\) is holomorphic on all of \(\A\).

We now give some examples of holomorphic functions for various algebras.

\begin{example}[The real numbers]
\label{example:the_real_numbers}
The real numbers are commutative, so we don't need to distinguish between left and right algebraic directional derivatives. Since
	\[
		v\delta=\nabla_vf(x)
	\]
implies that
	\[
		\delta=\lim_{h\to0}\frac{f(x+hv)-f(x)}{hv}=\lim_{h\to0}\frac{f(x+h)-f(x)}{h}=f'(x)
	\]
we have that the algebraic directional derivative is just the usual derivative when \(v\neq0\). A function \(f\colon\R\to\R\) is holomorphic if and only if it is differentiable in the usual sense.
\end{example}

\begin{example}[The complex numbers]
\label{example:the_complex_numbers}
The complex numbers are commutative, so we don't need to distinguish between left and right algebraic directional derivatives. Since \(\C\) is a field we have that there exists at most one algebraic directional derivative of a function \(f\) in any nonzero direction. If \(f\colon\C\to\C\) is differentiable then \(f(x+iy)=u(x,y)+v(x,y)i\) for some differentiable functions \(u,v\colon\R^2\to\R\). If \(f\) is holomorphic in the sense of \autoref{definition:holomorphic_function} then we must have that the same value is an algebraic directional derivative of \(f\) in the directions \(1\) and \(i\). These algebraic directional derivatives are
	\begin{align*}
		\frac{\nabla_1f(x+yi)}{1} &= \lim_{h\to0}\frac{f(x+yi+h(1))-f(x+yi)}{h(1)} \\
		&= \lim_{h\to0}\frac{u(x+h,y)+v(x+h,y)i-u(x,y)-v(x,y)i}{h} \\
		&= \lim_{h\to0}\frac{u(x+h,y)-u(x,y)}{h}+i\lim_{h\to0}\frac{v(x+h,y)-v(x,y)}{h} \\
		&= \pdv{u}{x}+i\pdv{v}{x}
	\end{align*}
and
	\begin{align*}
		\frac{\nabla_if(x+yi)}{i} &= \lim_{h\to0}\frac{f(x+yi+hi)=f(x+yi)}{hi} \\
		&= -i\lim_{h\to0}\frac{u(x,y+h)+v(x,y+h)i-u(x,y)-v(x,y)i}{h} \\
		&= -i\left(\lim_{h\to0}\frac{u(x,y+h)-u(x,y+h)}{h}+i\lim_{h\to0}\frac{v(x,y+h)-v(x,y)}{h}\right) \\
		&= -i\left(\pdv{u}{y}+i\pdv{v}{y}\right) \\
		&= \pdv{v}{y}-i\pdv{u}{y}.
	\end{align*}
It is a fundamental result of complex analysis that if
	\[\pdv{u}{x}+i\pdv{v}{x}=\pdv{v}{y}-i\pdv{u}{y}\]
then all directional derivatives exist and agree\cite{stein2003}. Thus, our definition of holomorphic function becomes to the usual one from complex analysis when considered in \(\C\).

Setting the two fundamental directional derivatives equal gives a system of linear partial differential equations, the Cauchy-Riemann equations, the solutions to which are the holomorphic functions on \(\C\). These equations are
	\[
		\pdv{u}{x}=\pdv{v}{y}\text{ and }\pdv{v}{x}=-\pdv{u}{y}.
	\]
\end{example}

\begin{example}[The group ring of \(\Z/2\Z\)]
\label{example:the_group_ring_of_Z_mod_2Z}
Let \(\Z/2\Z\) be the cyclic group of order \(2\) with identity \(e\) and generator \(a\). We consider holomorphic functions on \(\A=\R[\Z/2\Z]\). Again this algebra is commutative so there is only one notion of algebraic directional derivative. Unlike our previous examples, \(\A\) is not a division ring. When we take the algebraic directional derivative in a nonzero direction \(v\in\A\), there may be \(0\) or infinitely many solutions to \(v\delta=\nabla_vf(x)\). If there are \(0\) solutions then we cannot have \(f\) holomorphic. If there are infinitely many then it is possible for \(f\) to be holomorphic.

Consider \(re+sa\in\A\) for some \(r,s\in\R\). Taking \(\set{e,a}\) as a basis, we have that multiplication by \(re+sa\) is given by the matrix
	\[
		\begin{bmatrix}r&s\\s&r\end{bmatrix}.
	\]
The determinant of this matrix is \(r^2-s^2=(r+s)(r-s)\), so if \(f\) is differentiable then there is a unique algebraic directional derivative of \(f\) in each direction except those which satisfy \(r=\pm s\).

In the case where \(r=-s\) we can take as a representative the vector \(e-a\). Let \(f\colon\A\to\A\) be given by \(f(xe+ya)=u(x,y)e+v(x,y)a\) for differentiable scalar functions \(u\) and \(v\). We would like \((e-a)\delta=\nabla_{e-a}f(xe+ya)\) to have infinitely many solutions. Observe that
	\begin{align*}
		\nabla_{e-a}f(xe+ya) &= \lim_{h\to0}\frac{f(xe+ya+h(e-a))-f(xe+ya)}{h} \\
		&= \lim_{h\to0}\frac{u(x+h,y-h)e+v(x+h,y-h)a-u(x,y)e-v(x,y)a}{h} \\
		&= \nabla_{e-a}u(x,y)e+\nabla_{e-a}v(x,y)a \\
		&= \left(\pdv{u}{x}-\pdv{u}{y}\right)e+\left(\pdv{v}{x}-\pdv{v}{y}\right)a.
	\end{align*}
Let \(\delta=pe+qa\) for some \(p,q\in\R\). The resulting system of equations for \(\delta\) is
	\[
		p-q=\pdv{u}{x}-\pdv{u}{y}
	\]
and
	\[
		-p+q=\pdv{v}{x}-\pdv{v}{y}.
	\]
This can only hold when
	\[
		\pdv{u}{x}-\pdv{u}{y}=-\pdv{v}{x}+\pdv{v}{y}.
	\]

In the case where \(r=s\) a similar analysis gives a system
	\[
		p+q=\pdv{u}{x}+\pdv{u}{y}
	\]
and
	\[
		p+q=\pdv{v}{x}+\pdv{v}{y},
	\]
implying that
	\[
		\pdv{u}{x}+\pdv{u}{y}=\pdv{v}{x}+\pdv{v}{y}.
	\]

Taken together, the system of linear PDEs
	\[
		\pdv{u}{x}-\pdv{u}{y}=-\pdv{v}{x}+\pdv{v}{y}
	\]
and
	\[
		\pdv{u}{x}+\pdv{u}{y}=\pdv{v}{x}+\pdv{v}{y}
	\]
is satisfied exactly when \(\pdv{u}{x}=\pdv{v}{y}\) and \(\pdv{u}{y}=\pdv{v}{x}\).

Along any line besides those spanned by \(e-a\) and \(e+a\), a holomorphic function must satisfy
	\[
		(re+sa)\delta=\nabla_{re+sa}f(xe+ya)
	\]
which means that we must have
	\[
		\begin{bmatrix}r&s\\s&r\end{bmatrix}\delta=J_f(xe+ya)\begin{bmatrix}r\\s\end{bmatrix}.
	\]
This means that
	\[
		\delta=\begin{bmatrix}r&s\\s&r\end{bmatrix}^{-1}\begin{bmatrix}\pdv{u}{x}&\pdv{u}{y}\\\pdv{v}{x}&\pdv{v}{y}\end{bmatrix}\begin{bmatrix}r\\s\end{bmatrix}
	\]
so
	\[
		\delta=\frac{1}{r^2-s^2}\begin{bmatrix}r&-s\\-s&r\end{bmatrix}\begin{bmatrix}\pdv{u}{x}&\pdv{u}{y}\\\pdv{v}{x}&\pdv{v}{y}\end{bmatrix}\begin{bmatrix}r\\s\end{bmatrix}.
	\]
It follows that
	\[
		\delta=\frac{1}{r^2-s^2}\begin{bmatrix}r^2\pdv{u}{x}-rs\pdv{v}{x}+rs\pdv{u}{y}-s^2\pdv{v}{y}\\-rs\pdv{u}{x}+r^2\pdv{v}{x}-s^2\pdv{u}{y}+rs\pdv{v}{y}\end{bmatrix}.
	\]
Using the necessary conditions that \(\pdv{u}{x}=\pdv{v}{y}\) and \(\pdv{u}{y}=\pdv{v}{x}\) we see that this reduces to
	\[
		\delta=\pdv{u}{x}e+\pdv{u}{y}a.
	\]
Since this value of \(\delta\) is also an algebraic directional derivative for \(f\) at \(xe+ya\) in the directions along the lines spanned by \(e-a\) and \(e+a\), it must be that this is the unique value which works for all directions. That is, a differentiable function \(f\colon\A\to\A\) is holomorphic if and only if it satisfies
	\[
		\pdv{u}{x}=\pdv{v}{y}\text{ and }\pdv{u}{y}=\pdv{v}{x},
	\]
which we may view as something like the Cauchy-Riemann equations for \(\R[\Z/2\Z]\).
\end{example}

Before moving on to the titular object of these notes and discussing holomorphic functions on the RPSions \(\J\), let us first make some general remarks about holomorphic functions over any finite-dimensional real algebra \(\A\).

We first examine the holomorphy of constant functions. Recall that \(x\in\A\) is said to \emph{annihilate} \(\A\) when for all \(y\in\A\) we have that \(xy=yx=0\).

\begin{proposition}
\label{proposition:constant_functions_holomorphic}
Let \(\A\) be a finite-dimensional real algebra. The following are equivalent:
	\begin{enumerate}
		\item All constant functions \(f\colon\A\to\A\) are holomorphic.
		\item No nonzero element of \(\A\) annihilates \(\A\).
	\end{enumerate}
\end{proposition}

\begin{proof}
Suppose that all constant functions \(f\colon\A\to\A\) are holomorphic and let \(\alpha\in\A\setminus\set{0}\). We show that \(\alpha\) cannot annihilate \(\A\). Fix a constant function \(f\colon\A\to\A\). We have that \(f\) is holomorphic, so there exists a unique \(\delta\in\A\) such that for each \(v\in\A\) we have that \(v\delta=\delta v=\nabla_vf(x)=0\). We already know a candidate for \(\delta\), namely \(0\), so there must be some \(v\in\A\) such that either \(v\alpha\neq0\) or \(\alpha v\neq0\).

Suppose instead that no nonzero element of \(\A\) annihilates \(\A\) and let \(f\colon\A\to\A\) be a constant function. This means that \(0\) is the unique \(\delta\in\A\) such that for all \(v\in\A\) we have \(v\delta=\delta v=\nabla_vf(x)=0\). Thus, \(f\) is holomorphic.
\end{proof}

Since \(\nabla_vf(x)=0\) whenever \(f\) is a constant function, the existence of a single nonzero element \(\alpha\in\A\) which annihilates \(\A\) means that no constant functions are holomorphic in \(\A\). That is, either all constant functions are holomorphic in \(\A\) (in which case they all have derivative \(0\)) or no constant function is holomorphic in \(\A\).

Certainly in \(\R\) and \(\C\) we have that all constant functions are holomorphic. In \(\R[\Z/2\Z]\) we also have that all constant functions are holomorphic. To see this, observe that left multiplication by \(re+sa\) can only be the zero linear transformation on \(\R[\Z/2\Z]\) when \(r=s=0\). Thus, \(\R[\Z/2\Z]\) has no nonzero elements which annihilate it and by the previous result indicates that all constant functions are holomorphic.

The holomorphy of the identity function also has a nice algebraic characterization. Recall that \(e\in\A\) is an \emph{identity element} for \(\A\) when \(ex=xe=x\) for all \(x\in\A\).

\begin{proposition}
\label{proposition:identity_function_holomorphic}
Let \(\A\) be a finite-dimensional real algebra. The following are equivalent:
	\begin{enumerate}
		\item The identity function \(f\colon\A\to\A\) given by \(f(x)=x\) is holomorphic.
		\item There exists an identity element \(e\in\A\).
	\end{enumerate}
\end{proposition}

\begin{proof}
Suppose that the identity function \(f\) is holomorphic with \(f'(x)=\delta\). For each \(v\in\A\) we have that \(v\delta=\delta v=\nabla_vf(x)=v\). This means that any candidate \(\delta\) must be an identity element for \(\A\). Conversely, an identity element \(e\in\A\) satisfies the condition to be \(f'(x)\) for each direction \(v\in\A\).
\end{proof}

We find that whenever the identity function \(f\) is holomorphic, its derivative is \(f'(x)=e\) where \(e\in\A\) is the (necessarily unique) two-sided identity element of \(\A\). The identity function is holomorphic in \(\R\), \(\C\), and \(\R[\Z/2\Z]\). More generally, the identity function is holomorphic in any group algebra, any monoid algebra, or indeed \(\R[\mathbf{A}]\) for any unital magma \(\mathbf{A}\).

Note also that if an identity element \(e\in\A\) exists and \(\A\) is not trivial as a real vector space, then no nonzero \(\alpha\in\A\) could annihilate \(\A\), for it would have to satisfy \(\alpha e=e\alpha=0\) and also \(\alpha e=e\alpha=\alpha\). Thus, all constant functions as well as the identity function are holomorphic in any real algebra with an identity element. Moreover, these functions have their expected derivatives in analogy with real and complex analysis.

Now that we have some handle on when constant functions and the identity function are holomorphic, we consider sums, scalar multiples, and products of functions, with the goal of understanding when polynomial functions are holomorphic.

Given an element \(\alpha\in\A\), let \(\lambda_\alpha\colon\A\to\A\) be given by \(\lambda_\alpha(x)=\alpha x\) and let \(\rho_\alpha\colon\A\to\A\) be given by \(\rho_\alpha(x)=x\alpha\). We say that \(\alpha\) is \emph{divisible} in \(\A\) when at least one of \(\lambda_\alpha\) and \(\rho_\alpha\) is invertible. Note that \(0\) can only be divisible when \(\A\) is the trivial real vector space.

\begin{proposition}
\label{proposition:sum_of_holomorphic_functions}
Suppose that \(\A\) is a finite-dimensional real algebra which has a divisible element \(\alpha\). Given holomorphic functions \(f,g\colon\A\to\A\) we have that the sum \(f+g\colon\A\to\A\) given by \((f+g)(x)=f(x)+g(x)\) is holomorphic with \((f+g)'(x)=f'(x)+g'(x)\).
\end{proposition}

\begin{proof}
An algebraic directional derivative \(\delta\) of \(f+g\) in the direction \(v\) must solve \(v\delta=\delta v=\nabla_v(f+g)(x)\). Since \(\nabla_v(f+g)(x)=\nabla_vf(x)+\nabla_vg(x)\) we consider solutions to \(v\delta=\delta v=\nabla_vf(x)+\nabla_vg(x)\). We already know that \(\delta=f'(x)+g'(x)\) is a candidate since \(vf'(x)=f'(x)v=\nabla_vf(x)\) and \(vg'(x)=g'(x)v=\nabla_vg(x)\).

Consider the direction \(\alpha\). Either \(\alpha\delta=\nabla_v(f+g)(x)\) has a unique solution or \(\delta\alpha=\nabla_v(f+g)(x)\) has a unique solution. In any case, this shows that our candidate \(f'(x)+g'(x)\) must indeed be the unique algebraic directional derivative in the direction \(\alpha\), which means that \(f+g\) is holomorphic with \((f+g)'(x)=f'(x)+g'(x)\).
\end{proof}

The identity element of a real algebra \(\A\) is always divisible, if it exists, so any real algebra with an identity element has that sums of holomorphic functions are holomorphic (with the expected derivative).

Scalar multiplication may be handled similarly.

\begin{proposition}
\label{proposition:scalar_multiple_holomorphic_function}
Suppose that \(\A\) is a finite-dimensional real algebra which has a divisible element \(\alpha\), that \(f\colon\A\to\A\) is holomorphic, and that \(c\in\R\). We have that the function \(cf\colon\A\to\A\) given by \((cf)(x)=cf(x)\) is holomorphic with \((cf)'(x)=cf'(x)\).
\end{proposition}

\begin{proof}
Since \(f\) is holomorphic there exists a unique \(\delta\) (namely \(f'(x)\)) such that for all \(v\in\A\) we have \(v\delta=\delta v=\nabla_vf(x)\). Since \(\nabla_v(cf)(x)=c\nabla_vf(x)\) we find that \(cf'(x)\) is a solution to \(v\delta=\delta v=\nabla_v(cf)(x)\). Since we can take \(v\) to be the divisible element \(\alpha\), this must be the only solution which works for all \(v\).
\end{proof}

The preceding argument almost goes through without the assumption of a divisible element, but the case \(c=0\) prevents us from concluding that any solution to \(v\delta=\delta v=\nabla_vf(x)\) gives us a solution to \(v\delta=\delta v=\nabla_v(cf)(x)\) and vice versa. In any case, we now see that, at least under the assumption of a divisible element, we have that the holomorphic functions on \(\A\) form a real vector space under the usual addition and scalar multiplication.

Before proceeding to products of functions, we generalize from \autoref{proposition:identity_function_holomorphic} to consider the holomorphy of functions of the form \(\lambda_a\colon\A\to\A\) given by \(\lambda_a(x)=ax\). We call \(a\in\A\) \emph{central} when for all \(x\in\A\) we have that \(ax=xa\). That is, a central element commutes with all other elements under multiplication.

\begin{proposition}
\label{proposition:scalar_multiplication_holomorphic}
Suppose that \(\A\) contains a divisible element and let \(a\in\A\). We have that \(\lambda_a\colon\A\to\A\) is holomorphic if and only if \(a\) is central.
\end{proposition}

\begin{proof}
Note that \(\nabla_v\lambda_a(x)=av\) so \(v\delta=\delta v=av\) must hold in order for \(\delta\) to be an algebraic directional derivative for \(\lambda_a\). Certainly a solution to \(\delta v=av\) is \(\delta=a\). Since we can take \(v\) to be a divisible element we have that this is the only choice of \(\delta\) that works for all \(v\). It follows that \(va=av\) for all \(v\), so \(a\) is central.

On the other hand, a central element \(a\) would certainly solve \(v\delta=\delta v=av\) for \(\delta\). This candidate would then be the only one which works for all \(v\) under the assumption that \(\A\) contains a divisible element.
\end{proof}

Note that under the assumption that \(\A\) has a divisible element we find that for each central \(a\in\A\) we have that \(\lambda_a=\rho_a\) and that this function is holomorphic with \(\lambda_a'(x)=a\) for all \(x\in\A\). In the case that an identity element exists, the preceding result specializes to \autoref{proposition:identity_function_holomorphic} when we consider the identity function \(\lambda_e(x)=x\).

In order to state our result on the Product Rule, we need two more algebraic notions. Given \(a\in\A\) we say that \(a\) is \emph{middle nuclear} when for all \(x,y\in\A\) we have that \(x(ay)=(xa)y\) and we say that \(a\in\A\) is \emph{communal} when for all \(x,y\in\A\) we have that \(x(ay)=(ay)x\).

\begin{proposition}
\label{proposition:product_of_holomorphic_functions}
Suppose that \(\A\) contains a divisible element. The following are equivalent:
	\begin{enumerate}
		\item The Product Rule holds in \(\A\). That is, for all holomorphic functions \(f,g\colon\A\to\A\) we have that the function \(fg\colon\A\to\A\) given by \((fg)(x)=f(x)g(x)\) is again holomorphic with \((fg)'(x)=f'(x)g(x)+f(x)g'(x)\).
		\item Every central element of \(\A\) is middle nuclear and communal.
	\end{enumerate}
\end{proposition}

\begin{proof}
Suppose first that the Product Rule holds in \(\A\) and that \(a\in\A\) is central. Let \(b\in\A\). Since \(\A\) contains a divisible element we have that constant functions are all holomorphic with derivative \(0\). Let \(f(x)=ax\) and let \(g(x)=b\). Since the Product Rule holds we have that
	\[
		(fg)'(x)=f'(x)g(x)+f(x)g'(x)=ab+(ax)0=ab.
	\]
This quantity must be a solution \(\delta\) to
	\[
		v\delta=\delta v=\nabla_v(fg)(x)=\left(\nabla_vf(x)\right)g(x)+f(x)\nabla_vg(x).
	\]
We know that \(\nabla_vf(x)=vf'(x)=va\) and that \(\nabla_vg(x)=vg'(x)=0\) so we find that
	\[
		v(ab)=(ab)v=(va)b+(ax)0=(va)b.
	\]
Since \(b\) and \(v\) were taken arbitrarily and we have seen that \(v(ab)=(va)b\), the element \(a\) must be middle nuclear. Similarly, we have seen that \(v(ab)=(ab)v\), so \(a\) must be communal.

For the other direction, suppose that every central element of \(\A\) is middle nuclear and communal. Let \(f,g\colon\A\to\A\) be holomorphic. Note that, as before, \(\nabla_vf(x)=vf'(x)\) and \(\nabla_vg(x)=vg'(x)\). Moreover, \(f'(x)\) and \(g'(x)\) are central as they must solve \(v\delta=\delta v\) for all \(v\). We consider solutions to
	\[
		v\delta=\delta v=\nabla_v(fg)(x)=\left(\nabla_vf(x)\right)g(x)+f(x)\nabla_vg(x).
	\]
This means we want to solve
	\[
		v\delta=\delta v=(vf'(x))g(x)+f(x)(vg'(x)).
	\]
Since \(f'(x)\) and \(g'(x)\) are central, they must be middle nuclear and communal by assumption, so
	\begin{align*}
		(vf'(x))g(x)+f(x)(vg'(x)) &= v(f'(x)g(x))+f(x)(g'(x)v) \\
		&= v(f'(x)g(x))+(f(x)g'(x))v \\
		&= v(f'(x)g(x)+f(x)g'(x)) \\
		&= (f'(x)g(x)+f(x)g'(x))v.
	\end{align*}
This means that \(f'(x)g(x)+f(x)g'(x)\) is a solution \(\delta\) to \(v\delta=\delta v=\nabla_v(fg)(x)\). Since \(\A\) contains a divisible element, it must be the only one.
\end{proof}

We can make a stronger statement when \(\A\) has an identity element.

\begin{corollary}
\label{corollary:product_of_holomorphic_functions_identity_element}
Suppose that \(\A\) contains an identity element. The following are equivalent:
	\begin{enumerate}
		\item The Product Rule holds in \(\A\).
		\item The real algebra \(\A\) is commutative and associative.
	\end{enumerate}
\end{corollary}

\begin{proof}
Suppose that the Product Rule holds in \(\A\) and let \(e\) be the identity element of \(\A\). We have by \autoref{proposition:product_of_holomorphic_functions} that every central element of \(\A\) is middle nuclear and communal. The element \(e\) is central and is therefore communal. This means that for all \(x,y\in\A\) we have that \(x(ey)=(ey)x\) and thus \(xy=yx\). This shows that \(\A\) is commutative, whence all elements are central. Every element of \(\A\) is therefore middle nuclear, which means that for all \(x,y,z\in\A\) we have that \(x(yz)=(xy)z\). That is, \(\A\) is associative.

Suppose now that \(\A\) is a unital, commutative, associative real algebra. Certainly \(\A\) contains a divisible element (the identity) and all central elements of \(\A\) are middle nuclear and communal since all elements of \(\A\) are middle nuclear and communal. We have by \autoref{proposition:product_of_holomorphic_functions} that the Product Rule holds in \(\A\).
\end{proof}

When the conditions in \autoref{corollary:product_of_holomorphic_functions_identity_element} hold, we have that all constants as well as the identity function are holomorphic. We also have that the set of holomorphic functions is closed under addition and multiplication. It follows that all polynomial functions \(f(x)=\sum_{i=0}^na_ix^i\) with coefficients in \(\A\) are holomorphic and that their derivatives are as expected from experience with \(\R\) and \(\C\).

When the weaker conditions in \autoref{proposition:product_of_holomorphic_functions} hold, it takes more effort to describe the class of polynomials which are holomorphic. Let \(C_\A\) be the set of all central elements of a real algebra \(\A\) containing a divisible element. Suppose that in this algebra the Product Rule holds. We have that all constant functions are holomorphic, as are all the functions \(\lambda_a(x)=ax\) for some \(a\in C_\A\). We also know that the set of holomorphic functions is closed under addition and multiplication. Since we cannot assume that multiplication is associative in \(\A\), we can only state that the polynomials which are holomorphic in \(\A\) include those which can be inductively constructed from the constants and the \(\lambda_a\) by repeated addition and multiplication.

By our definition of an algebraic directional derivative, we must have that \(f'(x)\in C_\A\) for any holomorphic \(f\colon\A\to\A\) and any \(x\in\A\). There are two extremes where \autoref{proposition:product_of_holomorphic_functions} holds. One is where there are no central elements besides \(0\). In that case, the only functions which are differentiable are constants, which are certainly closed under multiplication. The other extreme is the case where every element is central. In that case \(\A\) is commutative and must be associative as well. Note that it is not necessary for \(\A\) to have an identity element in that case, but if it does then this is the situation described in \autoref{corollary:product_of_holomorphic_functions_identity_element}.

We would be remiss not to examine the Chain Rule in a generalization of elementary calculus. This rule also has an algebraic characterization and it is similar to that for the Product Rule. Given \(a\in\A\) we say that \(a\) is \emph{weakly middle nuclear} when for all \(x\in\A\) and all \(y\in C_\A\) we have that \(x(ay)=(xa)y\) and we say that \(a\in\A\) is \emph{weakly communal} when for all \(x\in\A\) and all \(y\in C_\A\) we have that \(x(ay)=(ay)x\).

\begin{proposition}
\label{proposition:chain_rule_holomorphic_functions}
Suppose that \(\A\) contains a divisible element. The following are equivalent:
	\begin{enumerate}
		\item The Chain Rule holds in \(\A\). That is, for all holomorphic functions \(f,g\colon\A\to\A\) we have that the function \(g\circ f\colon\A\to\A\) given by \((g\circ f)(x)=g(f(x))\) is again holomorphic with \((g\circ f)'(x)=f'(x)g'(f(x))\).
		\item Every central element of \(\A\) is weakly middle nuclear and weakly communal.
	\end{enumerate}
\end{proposition}

\begin{proof}
By the Chain Rule for the Jacobian we have that
	\begin{align*}
		\nabla_v(g\circ f)(x) &= J_{g\circ f}(x)v \\
		&= J_g(f(x))J_f(x)v \\
		&= J_g(f(x))\nabla_vf(x) \\
		&= \nabla_{\nabla_vf(x)}g(f(x)).
	\end{align*}
Since \(vf'(x)=\nabla_vf(x)\) and \(vg'(x)=\nabla_vg(x)\) we have that
	\[
		\nabla_v(g\circ f)(x)=\nabla_{vf'(x)}g(f(x))=(vf'(x))g'(f(x)).
	\]
If the Chain Rule holds then we must have that
	\[
		v(f'(x)g'(f(x)))=(vf'(x))g'(f(x))
	\]
Let \(a\in\A\) be central. We have that \(f(x)=ax\) is holomorphic. Let \(b\in C_\A\) and take \(g(x)=bx\), which is also holomorphic. By the Chain Rule we have for any \(v\in\A\) that
	\[
		v(ab)=(va)b,
	\]
so \(a\) is weakly middle nuclear. Since we have that
	\[
		v(f'(x)g'(f(x)))=(f'(x)g'(f(x)))v
	\]
it must be that
	\[
		v(ab)=(ab)v,
	\]
so \(a\) is weakly communal.

Suppose instead that every central element of \(\A\) is weakly middle nuclear and weakly communal. Since \(f'(x)\) satisfies \(vf'(x)=f'(x)v\) for all \(v\in\A\), we have that \(f'(x)\) is central and hence weakly middle nuclear and weakly communal by assumption. Similarly, \(g'(f(x))\) must be central in \(\A\). One solution to \(v\delta=\delta v=\nabla_v(g\circ f)(x)\) is \(\delta=f'(x)g'(f(x))\) for we have that
	\[
		v\delta=v(f'(x)g'(f(x)))=(vf'(x))g'(f(x))=\nabla_v(g\circ f)(x)
	\]
and
	\[
		\delta v=(f'(x)g'(f(x)))v=v(f'(x)g'(f(x)))=(vf'(x))g'(f(x))=\nabla_v(g\circ f)(x).
	\]
Since \(\A\) contains a divisible element, this must be the only solution.
\end{proof}

Since the Chain Rule is equivalent to an algebraic condition which is a weaker version of the algebraic condition for the Product Rule, we see that, under the assumption of a divisible element, any real algebra where the Product Rule holds must satisfy the Chain Rule, as well.

Our final general remarks on holomorphic functions concern the analogue of the Cauchy-Riemann equations. In \autoref{example:the_group_ring_of_Z_mod_2Z} we found a system of linear partial differential equations which were both necessary and sufficient for a differentiable function on \(\R[\Z/2\Z]\) to be holomorphic in that group algebra. In order to do this, we found those directions in which \(\lambda_{re+sa}\) was singular and used that the system \((re+sa)\delta=\nabla_{re+sa}f(xe+ya)\) must have at least one solution in each such direction. It could be difficult to replicate this process for an arbitrary real algebra as the problem of finding these directions is nonlinear.

Under our usual assumption of a divisible element, there is a better way to find these equations which play the role of the Cauchy-Riemann equations.

\begin{theorem}
\label{theorem:Cauchy-Riemann_equations}
Let \(\A\) be an \(n\)-dimensional real algebra containing a divisible element \(\alpha\) and let \(\psi\colon\A\to\A\) be either \(\lambda_\alpha^{-1}\) or \(\rho_\alpha^{-1}\), whichever exists. Suppose that \(B=\set{b_1,\dots,b_n}\) is a basis for \(\A\). Fix an inner product on \(\A\) by taking the bilinear extension of \(b_i\cdot b_j=1\) when \(i=j\) and \(0\) otherwise. Let \(f\colon\A\to\A\) be a differentiable function given by \(f(x)\cdot b_i=u_i(x)\) for differentiable functions \(u_i\colon\A\to\R\). We have that \(f\) is holomorphic if and only if for all \(1\le i,j\le n\) we have that
	\[
		b_i(\psi J_f(x)\alpha)\cdot b_j=(\psi J_f(x)\alpha)b_i\cdot b_j=\pdv{u_j}{x_i}.
	\]
\end{theorem}

\begin{proof}
Since a holomorphic \(f\) must satisfy \(\alpha\delta=\delta\alpha=\nabla_\alpha f(x)\) we must have that \(\delta=\psi J_f(x)\alpha\). A holomorphic \(f\) must also satisfy \(v\delta=\delta v=\nabla_vf(x)\) so taking \(v=b_i\) and \(\delta\) as before we find that
	\[
		b_i(\psi J_f(x)\alpha)=(\psi J_f(x)\alpha)b_i=\nabla_{b_i}f(x)=J_f(x)b_i.
	\]
Taking the inner product with \(b_j\) yields
	\[
		b_i(\psi J_f(x)\alpha)\cdot b_j=(\psi J_f(x)\alpha)b_i\cdot b_j=J_f(x)b_i\cdot b_j=\pdv{u_j}{x_i},
	\]
so at least these conditions on the \(u_i\) are necessary for \(f\) to be holomorphic.

We now show that they are sufficient. For each \(v\in\A\) we must show that \(\psi J_f(x)\alpha\) solves \(v\delta=\delta v=\nabla_vf(x)\) for \(\delta\). Expand \(v\) in the basis \(B\) as \(v=\sum_{i=1}^nv_ib_i\). Observe that
	\begin{align*}
		v(\psi J_f(x)\alpha)\cdot b_j &= \left(\sum_{i=1}^nv_ib_i\right)(\psi J_f(x)\alpha)\cdot b_j \\
		&= \sum_{i=1}^nv_ib_i(\psi J_f(x)\alpha)\cdot b_j \\
		&= \sum_{i=1}^nv_iJ_f(x)b_i\cdot b_j \\
		&= \left(\sum_{i=1}^nJ_f(x)v_ib_i\right)\cdot b_j \\
		&= \left(J_f(x)\sum_{i=1}^nv_ib_i\right)\cdot b_j \\
		&= (J_f(x)v)\cdot b_j \\
		&= \nabla_vf(x)\cdot b_j
	\end{align*}
so it must be that \(v(\psi J_f(x)\alpha)=\nabla_vf(x)\). A symmetric calculation applies for \((\psi J_f(x)\alpha)v=\nabla_vf(x)\). This shows that \(\psi J_f(x)\alpha\) is a candidate for \(f'(x)\) and since \(\A\) contains a divisible element it must be the only one.
\end{proof}

We provide some examples of the systems of linear PDEs produced by this method. As expected, the Cauchy-Riemann equations are such a system.

\begin{example}[The Cauchy-Riemann equations for \(\C\)]
\label{example:Cauchy-Riemann_equations_for_complex_numbers}
Take the divisible element \(\alpha\) to be \(1\) and let \(B=\set{1,i}\) be our basis. Write \(f(x+yi)=u(x,y)+v(x,y)i\) so that
	\[
		J_f(x+yi)=\begin{bmatrix}\pdv{u}{x}&\pdv{u}{y}\\\pdv{v}{x}&\pdv{v}{y}\end{bmatrix}.
	\]
We have that \(\psi\colon\C\to\C\) is the identity map in this case. We see that
	\[
		b_1(\psi J_f(x)\alpha)=1J_f(x)\begin{bmatrix}1\\0\end{bmatrix}=\pdv{u}{x}+\pdv{v}{x}i
	\]
so our first two equations are
	\[
		\left(\pdv{u}{x}+\pdv{v}{x}i\right)\cdot1=\pdv{u_1}{x_1}=\pdv{u}{x}
	\]
and
	\[
		\left(\pdv{u}{x}+\pdv{v}{x}i\right)\cdot i=\pdv{u_2}{x_1}=\pdv{v}{x}.
	\]
We have that
	\[
		b_2(\psi J_f(x)\alpha)=iJ_f(x)\begin{bmatrix}1\\0\end{bmatrix}=i\left(\pdv{u}{x}+\pdv{v}{x}i\right)=-\pdv{v}{x}+\pdv{u}{x}i
	\]
so our latter two equations are
	\[
		\left(-\pdv{v}{x}+\pdv{u}{x}i\right)\cdot1=\pdv{u_1}{x_2}=\pdv{u}{y}
	\]
and
	\[
		\left(-\pdv{v}{x}+\pdv{u}{x}i\right)\cdot i=\pdv{u_2}{x_2}=\pdv{v}{y}.
	\]
We therefore obtain the following system of linear PDEs:
	\begin{align*}
		\pdv{u}{x} &= \pdv{u}{x} \\
		\pdv{v}{x} &= \pdv{v}{x} \\
		-\pdv{v}{x} &= \pdv{u}{y} \\
		\pdv{u}{x} &= \pdv{v}{y}
	\end{align*}
While the first two of these equations are trivial, the last two are precisely the Cauchy-Riemann equations.
\end{example}

Let us now revisit the group ring \(\R[\Z/2\Z]\).

\begin{example}[The Cauchy-Riemann equations for \(\R{[\Z/2\Z]}\)]
\label{example:Cauchy-Riemann_equations_for_Z_mod_2Z}
Consider the group ring from \autoref{example:the_group_ring_of_Z_mod_2Z}. Take the divisible element \(\alpha\) to be \(e\) and let \(B=\set{e,a}\) be our basis. Write \(f(xe+ya)=u(x,y)e+v(x,y)a\) so that
	\[
		J_f(xe+ya)=\begin{bmatrix}\pdv{u}{x}&\pdv{u}{y}\\\pdv{v}{x}&\pdv{v}{y}\end{bmatrix}.
	\]
We have that \(\psi\) is again the identity map. We see that
	\[
		b_1(\psi J_f(x)\alpha)=eJ_f(x)\begin{bmatrix}1\\0\end{bmatrix}=\pdv{u}{x}e+\pdv{v}{x}a
	\]
so our first two equations are
	\[
		\left(\pdv{u}{x}e+\pdv{v}{x}a\right)\cdot e=\pdv{u_1}{x_1}=\pdv{u}{x}
	\]
and
	\[
		\left(\pdv{u}{x}e+\pdv{v}{x}a\right)\cdot a=\pdv{u_2}{x_1}=\pdv{v}{x}.
	\]
We have that
	\[
		b_2(\psi J_f(x)\alpha)=aJ_f(x)\begin{bmatrix}1\\0\end{bmatrix}=a\left(\pdv{u}{x}e+\pdv{v}{x}a\right)=\pdv{v}{x}e+\pdv{u}{x}a
	\]
so our latter two equations are
	\[
		\left(\pdv{v}{x}e+\pdv{u}{x}a\right)\cdot e=\pdv{u_1}{x_2}=\pdv{u}{y}
	\]
and
	\[
		\left(\pdv{v}{x}e+\pdv{u}{x}a\right)\cdot a=\pdv{u_2}{x_2}=\pdv{v}{y}.
	\]
We therefore obtain the following system of linear PDEs:
	\begin{align*}
		\pdv{u}{x} &= \pdv{u}{x} \\
		\pdv{v}{x} &= \pdv{v}{x} \\
		\pdv{v}{x} &= \pdv{u}{y} \\
		\pdv{u}{x} &= \pdv{v}{y}
	\end{align*}
While the first two of these equations are trivial, the last two are precisely the equations we obtained in \autoref{example:the_group_ring_of_Z_mod_2Z}.
\end{example}

\section{Holomorphic functions on the RPSions}
\label{section:holomorphic_functions_on_the_RPSions}
We can now specialize our knowledge of holomorphic functions on general finite-dimensional real algebras to the specific case of \(\J\). Note that, taking \((r,p,s)\) as an ordered basis, we have that
	\[
		\lambda_{ar+bp+cs}=\rho_{ar+bp+cs}=\begin{bmatrix}a+c&0&a\\b&a+b&0\\0&c&b+c\end{bmatrix}.
	\]
Computing
	\begin{align*}
		\det(\lambda_{ar+bp+cs}) &= \begin{vmatrix}a+c&0&a\\b&a+b&0\\0&c&b+c\end{vmatrix} \\
		&= (a+c)(a+b)(b+c)+abc
	\end{align*}
we see that \(\lambda_{ar+bp+cs}\) is typically invertible. That is, a typical element of \(\J\) is divisible. Since \(\J\) contains a divisible element, all our results about holomorphic functions apply to the RPSions.

The only element of \(\J\) which annihilates \(\J\) is \(0\). This is because such an element would have to satisfy
	\[
		\begin{bmatrix}a+c&0&a\\b&a+b&0\\0&c&b+c\end{bmatrix}=\begin{bmatrix}0&0&0\\0&0&0\\0&0&0\end{bmatrix}
	\]
so \(a=b=c=0\) is the only possible situation. By \autoref{proposition:constant_functions_holomorphic}, all constant functions on \(\J\) are holomorphic.

There is no identity element in \(\J\), for such an element would have to satisfy
	\[
		\begin{bmatrix}a+c&0&a\\b&a+b&0\\0&c&b+c\end{bmatrix}=\begin{bmatrix}1&0&0\\0&1&0\\0&0&1\end{bmatrix}
	\]
so \(a=b=c=0\) yet \(a+c=a+b=b+c=1\), which is impossible. By \autoref{proposition:identity_function_holomorphic}, this means that the identity function is not holomorphic in \(\J\).

By \autoref{proposition:sum_of_holomorphic_functions} and \autoref{proposition:scalar_multiple_holomorphic_function} we have that the holomorphic functions on \(\J\) form a real vector space in the usual way. Since \(\J\) is commutative all of its elements are central, which means that all functions \(\lambda_\alpha\colon\A\to\A\) are holomorphic by \autoref{proposition:scalar_multiplication_holomorphic}.

Since all elements of \(\J\) are central, the notions of middle nuclear and weakly middle nuclear coincide, as do the notions of communal and weakly communal. It follows from \autoref{proposition:product_of_holomorphic_functions} and \autoref{proposition:chain_rule_holomorphic_functions} that the Product Rule and Chain Rule hold in \(\J\) if and only if each element is middle nuclear and communal. Certainly each element is communal since \(\J\) is commutative, but the only middle nuclear element of \(\J\) is \(0\). To see this, suppose that \(ar+bp+cs\) is middle nuclear in \(\J\). We have that
	\[
		(r(ar+bp+cs))p=r((ar+bp+cs)p)
	\]
so
	\[
		((a+c)r+bp)p=r((a+b)p+cs)
	\]
and thus
	\[
		(a+b+c)p=cr+(a+b)p.
	\]
This means that \(c=0\). Similar calculations show that we must have \(a=0\) and \(b=0\), so \(0\) is the only middle nuclear element. Neither the Product Rule nor the Chain Rule hold in \(\J\). Indeed, few pairs of holomorphic functions \(f,g\colon\J\to\J\) should satisfy either of these rules, for middle nuclear elements are in very short supply in \(\J\).

We now apply \autoref{theorem:Cauchy-Riemann_equations} in order to obtain the analogue of the Cauchy-Riemann equations for \(\J\). Take the divisible element \(\alpha\) to be \(r+p+s\) and let \(\set{r,p,s}\) be our basis. Write
	\[
		f(xr+yp+zs)=u(x,y,z)r+v(x,y,z)p+w(x,y,z)s
	\]
so that
	\[
		J_f(xr+yp+zs)=\begin{bmatrix}\pdv{u}{x}&\pdv{u}{y}&\pdv{u}{z}\\\pdv{v}{x}&\pdv{v}{y}&\pdv{v}{z}\\\pdv{w}{x}&\pdv{w}{y}&\pdv{w}{z}\end{bmatrix}.
	\]
We have that \(\psi\colon\J\to\J\) is
	\[
		\frac{1}{9}\begin{bmatrix}4&1&-2\\-2&4&1\\1&-2&4\end{bmatrix}
	\]
so
	\begin{align*}
		&\psi J_f(xr+yp+zs)\alpha = \\ &\frac{1}{9}\left(4\left(\pdv{u}{x}+\pdv{u}{y}+\pdv{u}{z}\right)+\left(\pdv{v}{x}+\pdv{v}{y}+\pdv{v}{z}\right)-2\left(\pdv{w}{x}+\pdv{w}{y}+\pdv{w}{z}\right)\right)r+ \\
		&\frac{1}{9}\left(-2\left(\pdv{u}{x}+\pdv{u}{y}+\pdv{u}{z}\right)+4\left(\pdv{v}{x}+\pdv{v}{y}+\pdv{v}{z}\right)+\left(\pdv{w}{x}+\pdv{w}{y}+\pdv{w}{z}\right)\right)p+ \\
		&\frac{1}{9}\left(\left(\pdv{u}{x}+\pdv{u}{y}+\pdv{u}{z}\right)-2\left(\pdv{v}{x}+\pdv{v}{y}+\pdv{v}{z}\right)+4\left(\pdv{w}{x}+\pdv{w}{y}+\pdv{w}{z}\right)\right)s.
	\end{align*}
Since
	\begin{align*}
		&r(\psi J_f(xr+yp+zs)\alpha) = \\ &\frac{1}{9}\left(5\left(\pdv{u}{x}+\pdv{u}{y}+\pdv{u}{z}\right)-\left(\pdv{v}{x}+\pdv{v}{y}+\pdv{v}{z}\right)+2\left(\pdv{w}{x}+\pdv{w}{y}+\pdv{w}{z}\right)\right)r+ \\
		&\frac{1}{9}\left(-2\left(\pdv{u}{x}+\pdv{u}{y}+\pdv{u}{z}\right)+4\left(\pdv{v}{x}+\pdv{v}{y}+\pdv{v}{z}\right)+\left(\pdv{w}{x}+\pdv{w}{y}+\pdv{w}{z}\right)\right)p
	\end{align*}
our first three equations are
	\[
		\frac{1}{9}\left(5\left(\pdv{u}{x}+\pdv{u}{y}+\pdv{u}{z}\right)-\left(\pdv{v}{x}+\pdv{v}{y}+\pdv{v}{z}\right)+2\left(\pdv{w}{x}+\pdv{w}{y}+\pdv{w}{z}\right)\right)=\pdv{u}{x},
	\]
	\[
		\frac{1}{9}\left(-2\left(\pdv{u}{x}+\pdv{u}{y}+\pdv{u}{z}\right)+4\left(\pdv{v}{x}+\pdv{v}{y}+\pdv{v}{z}\right)+\left(\pdv{w}{x}+\pdv{w}{y}+\pdv{w}{z}\right)\right)=\pdv{v}{x},
	\]
and
	\[
		0=\pdv{w}{x}.
	\]
Since
	\begin{align*}
		&p(\psi J_f(xr+yp+zs)\alpha) = \\
		&\frac{1}{9}\left(2\left(\pdv{u}{x}+\pdv{u}{y}+\pdv{u}{z}\right)+5\left(\pdv{v}{x}+\pdv{v}{y}+\pdv{v}{z}\right)-\left(\pdv{w}{x}+\pdv{w}{y}+\pdv{w}{z}\right)\right)p+ \\
		&\frac{1}{9}\left(\left(\pdv{u}{x}+\pdv{u}{y}+\pdv{u}{z}\right)-2\left(\pdv{v}{x}+\pdv{v}{y}+\pdv{v}{z}\right)+4\left(\pdv{w}{x}+\pdv{w}{y}+\pdv{w}{z}\right)\right)s
	\end{align*}
our next three equations are
	\[
		0=\pdv{u}{y},
	\]
	\[
		\frac{1}{9}\left(2\left(\pdv{u}{x}+\pdv{u}{y}+\pdv{u}{z}\right)+5\left(\pdv{v}{x}+\pdv{v}{y}+\pdv{v}{z}\right)-\left(\pdv{w}{x}+\pdv{w}{y}+\pdv{w}{z}\right)\right)=\pdv{v}{y},
	\]
and
	\[
		\frac{1}{9}\left(\left(\pdv{u}{x}+\pdv{u}{y}+\pdv{u}{z}\right)-2\left(\pdv{v}{x}+\pdv{v}{y}+\pdv{v}{z}\right)+4\left(\pdv{w}{x}+\pdv{w}{y}+\pdv{w}{z}\right)\right)=\pdv{w}{y}.
	\]
Since
	\begin{align*}
		&s(\psi J_f(xr+yp+zs)\alpha) = \\ &\frac{1}{9}\left(4\left(\pdv{u}{x}+\pdv{u}{y}+\pdv{u}{z}\right)+\left(\pdv{v}{x}+\pdv{v}{y}+\pdv{v}{z}\right)-2\left(\pdv{w}{x}+\pdv{w}{y}+\pdv{w}{z}\right)\right)r+ \\
		&\frac{1}{9}\left(-\left(\pdv{u}{x}+\pdv{u}{y}+\pdv{u}{z}\right)+2\left(\pdv{v}{x}+\pdv{v}{y}+\pdv{v}{z}\right)+5\left(\pdv{w}{x}+\pdv{w}{y}+\pdv{w}{z}\right)\right)s
	\end{align*}
our final three equations are
	\[
		\frac{1}{9}\left(4\left(\pdv{u}{x}+\pdv{u}{y}+\pdv{u}{z}\right)+\left(\pdv{v}{x}+\pdv{v}{y}+\pdv{v}{z}\right)-2\left(\pdv{w}{x}+\pdv{w}{y}+\pdv{w}{z}\right)\right)=\pdv{u}{z},
	\]
	\[
		0=\pdv{v}{z},
	\]
and
	\[
		\frac{1}{9}\left(-\left(\pdv{u}{x}+\pdv{u}{y}+\pdv{u}{z}\right)+2\left(\pdv{v}{x}+\pdv{v}{y}+\pdv{v}{z}\right)+5\left(\pdv{w}{x}+\pdv{w}{y}+\pdv{w}{z}\right)\right)=\pdv{w}{z}.
	\]

This system can be further simplified. Three equations become the condition
	\[
		\pdv{u}{y}=\pdv{v}{z}=\pdv{w}{x}=0.
	\]
The remaining six equations, subject to this condition, become
	\[
		-4\pdv{u}{x}+5\pdv{u}{z}-\pdv{v}{x}-\pdv{v}{y}+2\pdv{w}{y}+2\pdv{w}{z}=0,
	\]
	\[
		-2\pdv{u}{x}-2\pdv{u}{z}-5\pdv{v}{x}+4\pdv{v}{y}+\pdv{w}{y}+\pdv{w}{z}=0,
	\]
	\[
		2\pdv{u}{x}+2\pdv{u}{z}+5\pdv{v}{x}-4\pdv{v}{y}-\pdv{w}{y}-\pdv{w}{z}=0,
	\]
	\[
		\pdv{u}{x}+\pdv{u}{z}-2\pdv{v}{x}-2\pdv{v}{y}-5\pdv{w}{y}+4\pdv{w}{z}=0,
	\]
	\[
		4\pdv{u}{x}-5\pdv{u}{z}+\pdv{v}{x}+\pdv{v}{y}-2\pdv{w}{y}-2\pdv{w}{z}=0,
	\]
and
	\[
		-\pdv{u}{x}-\pdv{u}{z}+2\pdv{v}{x}+2\pdv{v}{y}+5\pdv{w}{y}-4\pdv{w}{z}=0.
	\]
These six equations are equivalent to
	\[
		2\pdv{u}{z}=\pdv{u}{x}+\pdv{v}{y}-\pdv{w}{z},
	\]
	\[
		2\pdv{v}{x}=-\pdv{u}{x}+\pdv{v}{y}+\pdv{w}{z},
	\]
and
	\[
		2\pdv{w}{y}=\pdv{u}{x}-\pdv{v}{y}+\pdv{w}{z}.
	\]
Thus, an analogue of the Cauchy-Riemann equations for \(\J\) is the following system:
	\begin{align*}
		&\pdv{u}{y}=\pdv{v}{z}=\pdv{w}{x}=0 \\
		&2\pdv{u}{z}=\pdv{u}{x}+\pdv{v}{y}-\pdv{w}{z} \\
		&2\pdv{v}{x}=-\pdv{u}{x}+\pdv{v}{y}+\pdv{w}{z} \\
		&2\pdv{w}{y}=\pdv{u}{x}-\pdv{v}{y}+\pdv{w}{z}
	\end{align*}

Note that
	\[
		\pdv{u}{x}+\pdv{u}{z}=\frac{3}{2}\pdv{u}{x}+\frac{1}{2}\pdv{v}{y}-\frac{1}{2}\pdv{w}{z},
	\]
	\[
		\pdv{v}{x}+\pdv{v}{y}=-\frac{1}{2}\pdv{u}{x}+\frac{3}{2}\pdv{v}{y}+\frac{1}{2}\pdv{w}{z},
	\]
and
	\[
		\pdv{w}{y}+\pdv{w}{z}=\frac{1}{2}\pdv{u}{x}-\frac{1}{2}\pdv{v}{y}+\frac{3}{2}\pdv{w}{z}.
	\]
In the case where \(f\colon\J\to\J\) is holomorphic, we have that
	\begin{align*}
		&f'(xr+yp+zs)= \\
		&\psi J_f(xr+yp+zs)(r+p+s) \\
		&\frac{1}{9}\left(4\left(\pdv{u}{x}+\pdv{u}{z}\right)+\left(\pdv{v}{x}+\pdv{v}{y}\right)-2\left(\pdv{w}{y}+\pdv{w}{z}\right)\right)r+ \\
		&\frac{1}{9}\left(-2\left(\pdv{u}{x}+\pdv{u}{z}\right)+4\left(\pdv{v}{x}+\pdv{v}{y}\right)+\left(\pdv{w}{y}+\pdv{w}{z}\right)\right)p+ \\
		&\frac{1}{9}\left(\left(\pdv{u}{x}+\pdv{u}{z}\right)-2\left(\pdv{v}{x}+\pdv{v}{y}\right)+4\left(\pdv{w}{y}+\pdv{w}{z}\right)\right)s= \\
		&\left(\frac{1}{2}\pdv{u}{x}+\frac{1}{2}\pdv{v}{y}-\frac{1}{2}\pdv{w}{z}\right)r+ \\
		&\left(-\frac{1}{2}\pdv{u}{x}+\frac{1}{2}\pdv{v}{y}+\frac{1}{2}\pdv{w}{z}\right)p+ \\
		&\left(\frac{1}{2}\pdv{u}{x}-\frac{1}{2}\pdv{v}{y}+\frac{1}{2}\pdv{w}{z}\right)s= \\
		&\pdv{u}{z}r+\pdv{v}{x}p+\pdv{w}{y}s.
	\end{align*}

We can use these equations and this formula for \(f'\) to check some of our earlier work. Certainly all constants satisfy the Cauchy-Riemann system for \(\J\) given above since all their partial derivatives are \(0\). The formula for \(f'\) gives the anticipated value for the derivative, which is also \(0\). The function \(\lambda_{ar+bp+cs}\) has
	\[
		u(x,y,z)=(a+c)x+az,
	\]
	\[
		v(x,y,z)=bx+(a+b)y,
	\]
and
	\[
		w(x,y,z)=cy+(b+c)z.
	\]
Certainly
	\[
		\pdv{u}{y}=\pdv{v}{z}=\pdv{w}{x}=0.
	\]
We also have that
	\begin{align*}
		2\pdv{u}{z} &= 2a \\
		&= (a+c)+(a+b)-(b+c) \\
		&= \pdv{u}{x}+\pdv{v}{y}-\pdv{w}{z},
	\end{align*}
	\begin{align*}
		2\pdv{v}{x} &= 2b \\
		&= -(a+c)+(a+b)+(b+c) \\
		&= -\pdv{u}{x}+\pdv{v}{y}+\pdv{w}{z},
	\end{align*}
and
	\begin{align*}
		2\pdv{w}{y} &= 2c \\
		&= (a+c)-(a+b)+(b+c) \\
		&= \pdv{u}{x}-\pdv{v}{y}+\pdv{w}{z}
	\end{align*}
so \(\lambda_{ar+bp+cs}\) does indeed satisfy the system of PDEs. For the derivative we see that
	\begin{align*}
		\lambda_{ar+bp+cs}'(xr+yp+zs) &= \pdv{u}{z}r+\pdv{v}{x}p+\pdv{w}{y}s \\
		&= ar+bp+cs,
	\end{align*}
as expected.

The identity function \(f(xr+yp+zs)=xr+yp+zs\) is not holomorphic in \(\J\). Here we have that
	\[
		u(x,y,z)=x,
	\]
	\[
		v(x,y,z)=y,
	\]
and
	\[
		w(x,y,z)=z.
	\]
While it is the case that
	\[
		\pdv{u}{y}=\pdv{v}{z}=\pdv{w}{x}=0
	\]
note that
	\[
		2\pdv{u}{z}=2(0)=0\neq1=1+1-1=\pdv{u}{x}+\pdv{v}{y}-\pdv{w}{z},
	\]
so \(f\) cannot be holomorphic. Note that if it were holomorphic we would have
	\[
		f'(xr+yp+zs)=\pdv{u}{z}r+\pdv{v}{x}p+\pdv{w}{y}s=0,
	\]
yet the identity function is clearly nonconstant.

With a view towards the dynamical aspect of this paper, we consider the special case of the function \(f\colon\J\to\J\) given by \(f(x)=x^2\). Note that this makes sense even in a nonassociative real algebra such as \(\J\), although for higher powers of \(x\) we would need bracketing to avoid ambiguity. We claim that \(f\) is holomorphic with \(f'(x)=2x\). To see this directly from the definition, observe that
	\begin{align*}
		\nabla_vf(x) &= \lim_{h\to0}\frac{f(x+hv)-f(x)}{h} \\
		&= \lim_{h\to0}\frac{(x+hv)^2-x^2}{h} \\
		&= \lim_{h\to0}\frac{x^2+2hxv+h^2v^2-x^2}{h} \\
		&= \lim_{h\to0}(2xv+hv^2) \\
		&= 2xv.
	\end{align*}
Certainly \(\delta=2x\) is a candidate solution for \(v\delta=\delta v=2xv\) and since \(\J\) contains a divisible element it must be the only one which works for all \(v\).

We can also apply our knowledge of the Cauchy-Riemann system obtained previously to check that \(f(x)=x^2\) is holomorphic in \(\J\). We have that
	\[
		f(xr+yp+zs)=(x^2+2xz)r+(y^2+2xy)p+(z^2+2yz)s
	\]
so
	\[
		u(x,y,z)=x^2+2xz,
	\]
	\[
		v(x,y,z)=y^2+2xy,
	\]
and
	\[
		w(x,y,z)=z^2+2yz.
	\]
As usual, it is easy to check that
	\[
		\pdv{u}{y}=\pdv{v}{z}=\pdv{w}{x}=0.
	\]
Observe that
	\begin{align*}
		2\pdv{u}{z} &= 2(2x) \\
		&= (2x+2z)+(2y+2x)-(2z+2y) \\
		&= \pdv{u}{x}+\pdv{v}{y}-\pdv{w}{z},
	\end{align*}
	\begin{align*}
		2\pdv{v}{x} &= 2(2y) \\
		&= -(2x+2z)+(2y+2x)+(2z+2y) \\
		&= -\pdv{u}{x}+\pdv{v}{y}+\pdv{w}{z},
	\end{align*}
and
	\begin{align*}
		2\pdv{w}{y} &= 2(2z) \\
		&= (2x+2z)-(2y+2x)+(2z+2y) \\
		&= \pdv{u}{x}-\pdv{v}{y}+\pdv{w}{z}
	\end{align*}
so \(f\) is holomorphic in \(\J\). The derivative of \(f\) is
	\[
		f'(xr+yp+zs)=\pdv{u}{z}r+\pdv{v}{x}p+\pdv{w}{y}s=2xr+2yp+2zs=2(xr+yp+zs),
	\]
in accord with our previous calculation from the definition of the derivative.

Since the holomorphic functions on \(\J\) include all constants and are closed under addition, we have that \(x\mapsto x^2+c\) is holomorphic for all \(c\in\J\) with the expected derivative. It may seem a bit odd that all of these functions can be holomorphic while \(x\mapsto x\) is not, but that's the situation we find ourselves in.

\section{Dynamics in real algebras}
\label{section:dynamics_in_real_algebras}
The dynamics of a holomorphic function in an arbitrary finite-dimensional real algebra may be treated similarly to that for a holomorphic function on \(\C\). We supply some definitions. We will write \(f^0(x)=x\) and \(f^{n+1}(x)=f(f^n(x))\) for \(n\in\N\).

\begin{definition}[Orbit]
\label{definition:orbit}
Given a holomorphic function \(f\colon\A\to\A\) the \emph{orbit} of \(x\in\A\) under \(f\) is
	\[
		\set[f^n(x)\in\A]{n\in\N}.
	\]
\end{definition}

\begin{definition}[Periodic point]
\label{definition:periodic_point}
We say that \(x\in\A\) is \emph{periodic} under a holomorphic function \(f\colon\A\to\A\) when \(f^n(x)=x\) for some \(n>0\). We say that a periodic point \(x\) under \(f\) has \emph{period} \(n\) (or is \emph{\(n\)-periodic}) when the orbit of \(x\) under \(f\) contains exactly \(n\) elements.
\end{definition}

Unlike in the case of complex dynamics, we cannot define whether a cycle is attracting or repelling based on the derivative of \(f^n\) for the appropriate \(n\) because \(f^n\) may not be holomorphic even when \(f\) itself is. This difficulty may be circumvented as follows.

\begin{definition}[Periodic derivative]
\label{definition:periodic_derivative}
Given an \(n\)-periodic point \(x\in\A\) of a holomorphic function \(f\colon\A\to\A\) the \emph{periodic derivative} of \(x\) under \(f\) is
	\[
		\lambda^f_x=\lambda_{f'(f^{n-1}(x))}\lambda_{f'(f^{n-2}(x))}\cdots\lambda_{f'(f(x))}\lambda_{f'(x)}.
	\]
\end{definition}

This linear transformation is closely related to the behavior of a point near \(x\). Suppose that \(\epsilon\in\A\) is a point very close to \(0\). We have that
	\[
		f(x+\epsilon)\approx f(x)+f'(x)\epsilon.
	\]
It follows that
	\[
		f^2(x+\epsilon)\approx f(f(x)+f'(x)\epsilon)\approx f^2(x)+f'(f(x))(f'(x)\epsilon).
	\]
Continuing in this fashion, we see that
	\begin{align*}
		f^n(x+\epsilon) &\approx f(f^{n-1}(x)+f'(f^{n-2}(x))(\cdots f'(f(x))(f'(x)\epsilon))) \\
		&\approx f^n(x)+f'(f^{n-1}(x))(\cdots f'(f(x))(f'(x)\epsilon)) \\
		&= x+\lambda_x^f(\epsilon).
	\end{align*}

Let \(F\colon\R^k\to\R^k\) be a linear transformation and endow \(\R^k\) with the Euclidean norm. Recall that the \emph{matrix norm} of \(F\) is
	\[
		\norm{F}=\sup\set[\norm{Fv}\in\R]{\norm{v}\le1}.
	\]
This is equivalent to defining \(\norm{F}\) to be the square root of the largest eigenvalue of \(F^TF\).

\begin{definition}[Types of periodic points]
\label{definition:types_of_periodic_points}
Suppose that \(\A\) has underlying vector space \(\R^k\) some some \(k\in\N\), which we equip with the Euclidean norm. Let \(x\in\A\) be an \(n\)-periodic point of a holomorphic function \(f\colon\A\to\A\). We say that \(x\) is
	\begin{itemize}
		\item \emph{super-attracting} when \(\norm{\lambda_x^f}=0\),
		\item \emph{attracting} when \(0<\norm{\lambda_x^f}<1\),
		\item \emph{indifferent} when \(\norm{\lambda_x^f}=1\), and
		\item \emph{repelling} when \(\norm{\lambda_x^f}>1\).
	\end{itemize}
\end{definition}

Note that our definition of super-attracting is equivalent to \(\lambda_x^f=0\), for if \(\norm{\lambda_x^f}=0\) then for all \(v\) with \(\norm{v}\le1\) we have that \(\norm{\lambda_x^fv}\le0\). This means that \(\lambda_x^fv=0\).

These definitions comport with those for complex-valued functions\cite[p.91]{beardon1991}. Since \(\C\) is associative we have that
	\[
		\lambda_x^f=\lambda_{f'(f^{n-1}(x))\cdots f'(f(x))f'(x)}=\lambda_{(f^n)'(x)}.
	\]
We find that \(\norm{\lambda_x^f}=\norm{\lambda_{(f^n)'(x)}}=\abs{(f^n)'(x)}\), as expected.

\begin{example}[Fixed points of \(x\mapsto x^2\) in \(\J\)]
\label{example:fixed_points_of_squaring_map}
We classify the behavior of the fixed points of the map \(f\colon\J\to\J\) given by \(f(x)=x^2\). Suppose that \(ar+bp+cs\) is a fixed point of \(f\). We have that
	\[
		(ar+bp+cs)^2=ar+bp+cs
	\]
so it must be that
	\[
		(a^2+2ac)r+(b^2+2ab)p+(c^2+2bc)s=ar+bp+cs.
	\]
If \(a=0\) then we have
	\[
		b^2p+(c^2+2bc)s=bp+cs.
	\]
We must have \(b^2=b\), so either \(b=0\) or \(b=1\). When \(b=0\) we see that \(c^2s=cs\) so again either \(c=0\) or \(c=1\). We find that two fixed points are \(0\) and \(s\). When \(b=1\) we have that \((c^2+2c)s=cs\) so \(c^2+c=0\). Either \(c=0\) or \(c=-1\). We find that two fixed points are \(p\) and \(p-s\). By symmetry, we have found the fixed points \(0\), \(r\), \(p\), \(s\), \(r-p\), \(p-s\), and \(s-r\). This exhausts the cases where at least one of \(a\), \(b\), or \(c\) is \(0\).

If all of the components are nonzero, we may conclude that \(a+2c=1\), \(b+2a=1\), and \(c+2b=1\). The unique solution to this linear system is \(a=b=c=\frac{1}{3}\), so our final fixed point is \(\frac{1}{3}r+\frac{1}{3}p+\frac{1}{3}s\).

Since we are working with fixed points, we have in each case that
	\[
		\lambda_x^f=\lambda_{f'(x)}=\lambda_{2x}.
	\]
When \(x=ar+bp+cs\) this means that
	\[
		\lambda_x^f=\begin{bmatrix}2a+2c&0&2a\\2b&2a+2b&0\\0&2c&2b+2c\end{bmatrix}.
	\]
We have that \(\norm{\lambda_x^f}\) is the square root of the largest eigenvalue of \((\lambda_x^f)^T\lambda_x^f\) so we must find the eigenvalues of
	\begin{align*}
		(\lambda_x^f)^T\lambda_x^f &= \begin{bmatrix}2a+2c&2b&0\\0&2a+2b&2c\\2a&0&2b+2c\end{bmatrix}\begin{bmatrix}2a+2c&0&2a\\2b&2a+2b&0\\0&2c&2b+2c\end{bmatrix} \\
		&= 4\begin{bmatrix}(a+c)^2+b^2&b(a+b)&a(a+c)\\b(a+b)&(a+b)^2+c^2&c(b+c)\\a(a+c)&c(b+c)&(b+c)^2+a^2\end{bmatrix}.
	\end{align*}

When \(x=0\) we have that \(\lambda_x^f=0\) so \(0\) is a super-attracting fixed point of \(f\). When \(x=r\) we have that
	\[
		(\lambda_x^f)^T\lambda_x^f=\begin{bmatrix}4&0&4\\0&4&0\\4&0&4\end{bmatrix}.
	\]
The eigenvalues of this matrix are \(0\), \(4\), and \(8\), so \(\norm{\lambda_x^f}=2\sqrt{2}>1\) and thus \(r\) is a repelling fixed point of \(f\). By symmetry, \(p\) and \(s\) are also repelling fixed points of \(f\). When \(x=r-p\) we have that
	\[
		(\lambda_x^f)^T\lambda_x^f=\begin{bmatrix}8&0&4\\0&0&0\\4&0&8\end{bmatrix}.
	\]
The eigenvalues of this matrix are \(0\), \(4\), and \(12\), so \(\norm{\lambda_x^f}=2\sqrt{3}>1\) and thus \(r-p\) is a repelling fixed point of \(f\). By symmetry \(p-s\) and \(s-r\) are also repelling fixed points of \(f\). When \(x=\frac{1}{3}r+\frac{1}{3}p+\frac{1}{3}s\) we have that
	\[
		(\lambda_x^f)^T\lambda_x^f=\frac{4}{9}\begin{bmatrix}5&2&2\\2&5&2\\2&2&5\end{bmatrix}.
	\]
The eigenvalues of this matrix are \(\frac{4}{3}\) and \(4\), so \(\norm{\lambda_x^f}=2>1\) and thus \(\frac{1}{3}r+\frac{1}{3}p+\frac{1}{3}s\) is a repelling fixed point of \(f\).
\end{example}

Our final example is the beginning of our examination of the Mandelbrot set for the RPSions.

\begin{example}[Behavior of \(0\) under quadratic mappings]
\label{example:behavior_of_0}
Given \(\alpha\in\J\) let \(q_\alpha\colon\J\to\J\) be given by \(q_\alpha(x)=x^2+\alpha\). Suppose that \(0\) is \(n\)-periodic under the map \(q_\alpha\colon\J\to\J\) for some constant \(\alpha=ar+bp+cs\in\J\). Note that \(q_\alpha'(x)=2x\). We have that
	\[
		\lambda_0^{q_\alpha}=\lambda_{2q_\alpha^{n-1}(0)}\lambda_{2q_\alpha^{n-2}(0)}\cdots\lambda_{2q_\alpha(0)}\lambda_{2(0)}=0
	\]
so \(\norm{\lambda_0^{q_\alpha}}=0\) and \(0\) is a super-attracting periodic point under \(q_\alpha\). Intuitively, this means that points close to \(0\) should very quickly approach the orbit of \(0\) under iteration of \(q_\alpha\).

For which values of \(\alpha\) is \(0\) periodic? We have that \(0\) is a fixed point of \(q_\alpha\) when \(0^2+(ar+bp+cs)=0\), which only happens when \(ar+bp+cs=0\). We have that \(0\) belongs to a \(2\)-cycle under \(q_\alpha\) when
	\[
		(ar+bp+cs)^2+(ar+bp+cs)=0.
	\]
This means that
	\[
		(a^2+2ac+a)r+(b^2+2ab+b)p+(c^2+2bc+c)s=0.
	\]
If \(a=0\) then
	\[
		(b^2+b)p+(c^2+2bc+c)s=0.
	\]
If \(b=0\) as well then
	\[
		(c^2+c)s=0,
	\]
so either \(c=0\) (which is the case when \(0\) is a fixed point that we already considered) or \(c=-1\). We see that \(0\) belongs to a \(2\)-cycle when \(ar+bp+cs=-s\). By symmetry, \(0\) belongs to a \(2\)-cycle when \(ar+bp+cs=-r\) or \(ar+bp+cs=-p\). If \(b\neq0\) then we have that \(b=-1\) so
	\[
		c^2-2c+c=0.
	\]
It follows that \(c^2-c=0\) so either \(c=0\) (which is the case of \(-p\) which we already found) or \(c=1\). We see that \(0\) belongs to a \(2\)-cycle when \(ar+bp+cs=-p+s\). By symmetry, \(0\) belongs to a \(2\)-cycle when \(ar+bp+cs=-s+r\) or \(ar+bp+cs=-r+p\). This exhausts the cases where at least one of \(a\), \(b\), or \(c\) is \(0\).

If all of the components are nonzero, we find that \(a+2c+1=0\), \(b+2a+1=0\), and \(c+2b+1=0\). The unique solution to this linear system is \(a=b=c=\frac{-1}{3}\), so we have that \(0\) belongs to a \(2\)-cycle when \(ar+bp+cs=-\frac{1}{3}r-\frac{1}{3}p-\frac{1}{3}s\). Note that each of the values of \(\alpha=ar+bp+cs\) for which \(q_\alpha^2(0)=0\) is the negation of a value of \(x\) with \(x^2=x\) we found in \autoref{example:fixed_points_of_squaring_map}. Certainly if \(\alpha^2=\alpha\) then we have that
	\[
		q_{-\alpha}^2(0)=q_{-\alpha}(0^2-\alpha)=q_{-\alpha}(-\alpha)=(-\alpha)^2-\alpha=\alpha^2-\alpha=\alpha-\alpha=0,
	\]
so \(0\) is either a fixed point or belongs to a \(2\)-cycle. Our calculation shows that there are no other examples than these.
\end{example}

\section{Software for fractal animations}
\label{section:software_for_fractal_animations}
We are finally ready to define the Mandelbrot set for a real algebra \(\A\) and draw pictures of it.

\begin{definition}[Mandelbrot set over a real algebra]
\label{definition:Mandelbrot_set_over_a_real_algebra}
Given a finite-dimensional real algebra \(\A\) and \(\alpha\in\A\), let \(q_\alpha\colon\A\to\A\) be given by \(q_\alpha(x)=x^2+\alpha\). Suppose that \(\A\) has underlying vector space \(\R^k\) some some \(k\in\N\), which we equip with the Euclidean norm. The \emph{Mandelbrot set (over \(\A\))} is
	\[
		\set[\alpha\in\A]{\norm{q_\alpha^n(0)}\not\to\infty\text{ as }n\to\infty}.
	\]
\end{definition}

As when we defined attracting and repelling periodic points in \autoref{definition:types_of_periodic_points}, we need to make use of a norm on \(\A\). We will only consider the case where \(\A\) is in fact \(\R^k\) as a real vector space and our norm is the Euclidean norm. Our definition of a Mandelbrot set is a direct generalization of one of the equivalent characterizations of the Mandelbrot set over the complex numbers\cite[p.249]{devaney1992}.

We have already seen some points in the Mandelbrot set for \(\J\), under the identification of \(\R^3\) with the space spanned by \(\set{r,p,s}\) in the obvious way. These are the values of \(\alpha\) for which \(0\) is fixed or belongs to a \(2\)-cycle for \(q_\alpha\), which we found in \autoref{example:behavior_of_0}. These points are \(0\), \(-r\), \(-p\), \(-s\), \(-r+p\), \(-p+s\), \(-s+r\), and \(-\frac{1}{3}r-\frac{1}{3}p-\frac{1}{3}s\).

Consider a point \(\alpha\in\J\) such that \(\alpha^2=\alpha\). We found these in \autoref{example:fixed_points_of_squaring_map}. Given \(k\in\N\) we have that
	\[
		q_\alpha(k\alpha)=k^2\alpha^2+\alpha=(k^2+1)\alpha.
	\]
Computing \(q_\alpha^n(0)\) entails applying the rule \(k\mapsto k^2+1\) a total of \(n\) times, starting with \(k=0\). Certainly for any nonzero \(\alpha\) we have that \(\norm{k\alpha}\to\infty\) as \(k\to\infty\), so \(r\), \(p\), \(s\), \(r-p\), \(p-s\), \(s-r\), and \(\frac{1}{3}r+\frac{1}{3}p+\frac{1}{3}s\) all lie outside the Mandelbrot set for \(\J\).

More generally, given \(\beta,\gamma\in\R\) we have that
	\[
		q_{\beta\alpha}(\gamma\alpha)=\gamma^2\alpha^2+\beta\alpha=(\gamma^2+\beta)\alpha.
	\]
Computing \(q_{\beta\alpha}^n(0)\) entails applying the rule \(\gamma\mapsto\gamma^2+\beta\) a total of \(n\) times. Each line spanned by an \(\alpha\) with \(\alpha^2=\alpha\) therefore has the same behavior as the usual Mandelbrot set restricted to the real line, up to some scaling. This means that for any \(\beta\) with \(\beta<-2\) or \(\beta>\frac{1}{4}\) we have that \(\norm{q_{\beta\alpha}^n(0)}\to\infty\) as \(n\to\infty\)\cite[p.19]{beardon1991}. Thus, with the exception of a line segment containing the origin, the lines spanned by \(r\), \(p\), \(s\), \(r-p\), \(p-s\), \(s-r\), and \(\frac{1}{3}r+\frac{1}{3}p+\frac{1}{3}s\) all lie outside the Mandelbrot set for \(\J\).

Rather than continue to analytically study the orbit of \(0\) under \(q_\alpha\) for various choices of \(\alpha\), we now draw images of the Mandelbrot set using the code found at \href{https://codeberg.org/caten/RealAlgebraPlot}{https://codeberg.org/caten/RealAlgebraPlot}.

The core component of this software project is the C++ program \texttt{plot.cpp}, which draws a rectangular image of a Mandelbrot set for a finite-dimensional real algebra. A real algebra \(\A\) is supplied in the form of a text file. The first line is the dimension \(n\) of \(\A\) and the next \(n^3\) lines are the structure constants for its multiplication. That is, the subsequent lines may be identified with the at-most-three-digit numbers from \(0\) to \(n^3-1\) in base \(n\) in ascending order. The line corresponding to the base \(n\) numeral \(ijk\) contains a positive or negative decimal number which is the structure constant \(s_{ijk}\). The multiplication of \(\A\) is then
	\[
		\left(\sum_{i=0}^{n-1}u_{i+1}e_{i+1}\right)\left(\sum_{j=0}^{n-1}v_{j+1}e_{j+1}\right)=\sum_{k=0}^{n-1}\left(\sum_{i,j=0}^{n-1}s_{ijk}u_{i+1}v_{j+1}\right)e_{k+1}
	\]
where the \(u_i\) and \(v_j\) are real coefficients and the \(e_k\) are the standard basis vectors for \(\A\cong\R^n\) as a real vector space.

In the beginning of \texttt{plot.cpp} we use a macro to set the maximum dimension (\texttt{MAX\char`_DIMENSION}) of a real algebra to \(8\). It is easy enough to change this to a much larger number, if one desires. The reason for this restriction is that, if one were to use template metaprogramming to create a struct \texttt{Algebra<n>} for each \(n\), one would need to know at compile time, rather than runtime, which value of \(n\) would be used. Since we would like to be able to call this plotting command for various values of \(n\), it is a bit more elegant to simply define the \texttt{Algebra} struct to be essentially \texttt{Algebra<MAX\char`_DIMENSION>} and treat algebras of lower dimension as algebras of \texttt{MAX\char`_DIMENSION} where we ignore some extraneous coordinates. Viewing the structure constants as an \(n\times n\times n\) hypermatrix, this is basically padding out to a larger \(3\)-dimensional hypermatrix with entries we never read or write.

Necessarily the Mandelbrot set for \(\A\) is a set of points in \(n\)-dimensional space, so we must take a cross-section in order to draw a \(2\)-dimensional image. The program \texttt{plot.cpp} takes as arguments the upper and lower bounds on the coefficients of \(e_1\) and \(e_2\). We then fix the other \(n-2\) coordinates in order to obtain a \(2\)-dimensional cross-section. Let the upper and lower bounds for \(e_1\) be \(x_0\) and \(x_1\). Similarly, let the upper and lower bounds for \(e_2\) be \(y_0\) and \(y_1\). Call the remaining \(n-2\) coefficients \(z_1,\dots,z_{n-2}\).

We fix a step size \(s\) and we examine all points of the form
	\[
		\alpha=(x_0+ks)e_1+(y_0+ks)e_2+z_1e_3+\cdots+z_{n-2}e_n
	\]
where \(k\in\N\), \(x_0+ks<x_1\), and \(y_0+ks<y_1\). For each such \(\alpha\) we compute \(q_\alpha^m(0)\) for larger and larger \(m\) where \(q_\alpha\colon\A\to\A\) is given by \(q_\alpha(w)=w^2+\alpha\).

In order to avoid unnecessary calculations, we check at each application of \(q_\alpha\) whether
	\[
		\norm{q_\alpha^m(0)-q_\alpha^{m+1}(0)}
	\]
has become smaller than some \texttt{fixed\char`_threshold}. If so, we assume that we have reached a fixed point under \(q_\alpha\) and stop iterating. At each step we also check whether \(\norm{q_\alpha^m(0)}\) has become greater than some \texttt{bailout\char`_threshold}. If so, we assume that we have found an orbit which blows up to infinity under \(q_\alpha\).

Of course, we can't iterate forever, so we fix a maximum number of iterations of \(q_\alpha\) to attempt. We use this to assign a natural number to each point \(\alpha\) in our selected range. If we are on iteration \(m\) and the bailout threshold is passed, we assign to that point the number \(m\). If the fixed threshold is passed, we assign to that point the number \(0\). If we iterate until the maximum allowed number of iterations and neither condition occurs, we also assign to that point the number \(0\). The number \(0\) means the point appears to be in the Mandelbrot set for \(\A\). Any other number is a measurement of how many steps it took for the critical orbit to ``escape to infinity''.

The program offers three coloring schemes. The ``monochrome'' option colors the Mandelbrot set black and the other points white, the ``greyscale'' option colors the Mandelbrot set black and colors the escaping points by various shades of grey, depending on how quickly they passed the bailout threshold, and the ``rainbow'' option behaves similarly to the greyscale one, but with bands of color instead of shades of grey. These three options are shown in \autoref{figure:coloring_options}.

\begin{figure}
	\begin{center}
		\begin{tabular}{ccc}
			\includegraphics[height=3.5cm]{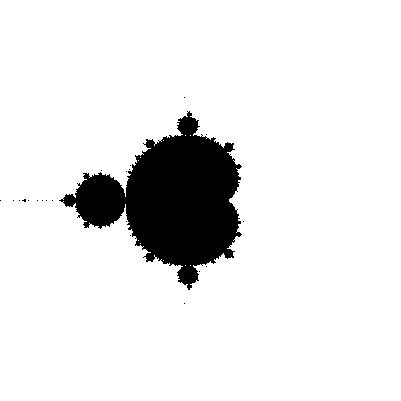} & \includegraphics[height=3.5cm]{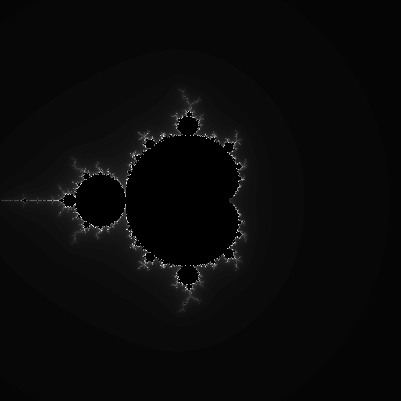} & \includegraphics[height=3.5cm]{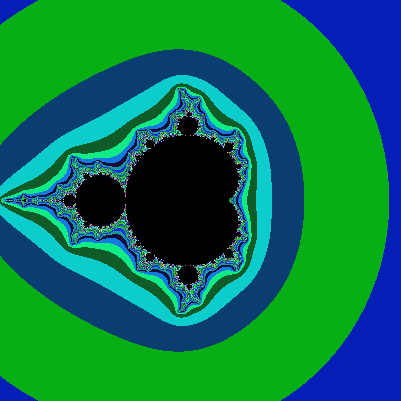}
		\end{tabular}
	\end{center}
	\caption{Coloring options for plotting fractals}
	\label{figure:coloring_options}
\end{figure}

As it can be a bit cumbersome to manually pass numerous command-line arguments to a compiled C++ program, Lua scripts are provided which automate the use of the basic plotting faculty. The script \texttt{real\char`_algebra\char`_plot.lua} contains a more user-friendly interface to \texttt{plot.cpp} by supplying many arguments as defaults. An animation function is also provided. This takes the same arguments as the usual plotting function, but requests a minimum and maximum value for the \(e_3\) coefficient as well. Motion in this third dimension is represented as time in a resulting animation, with one frame for each step taken in this third coordinate direction. See \autoref{figure:RPS_Mandelbrot_slices} for examples of the kinds of \(2\)-dimensional slices which may be concatenated into an animation by the program.

\begin{figure}
	\begin{center}
		\begin{tabular}{cc}
			\includegraphics[height=5cm]{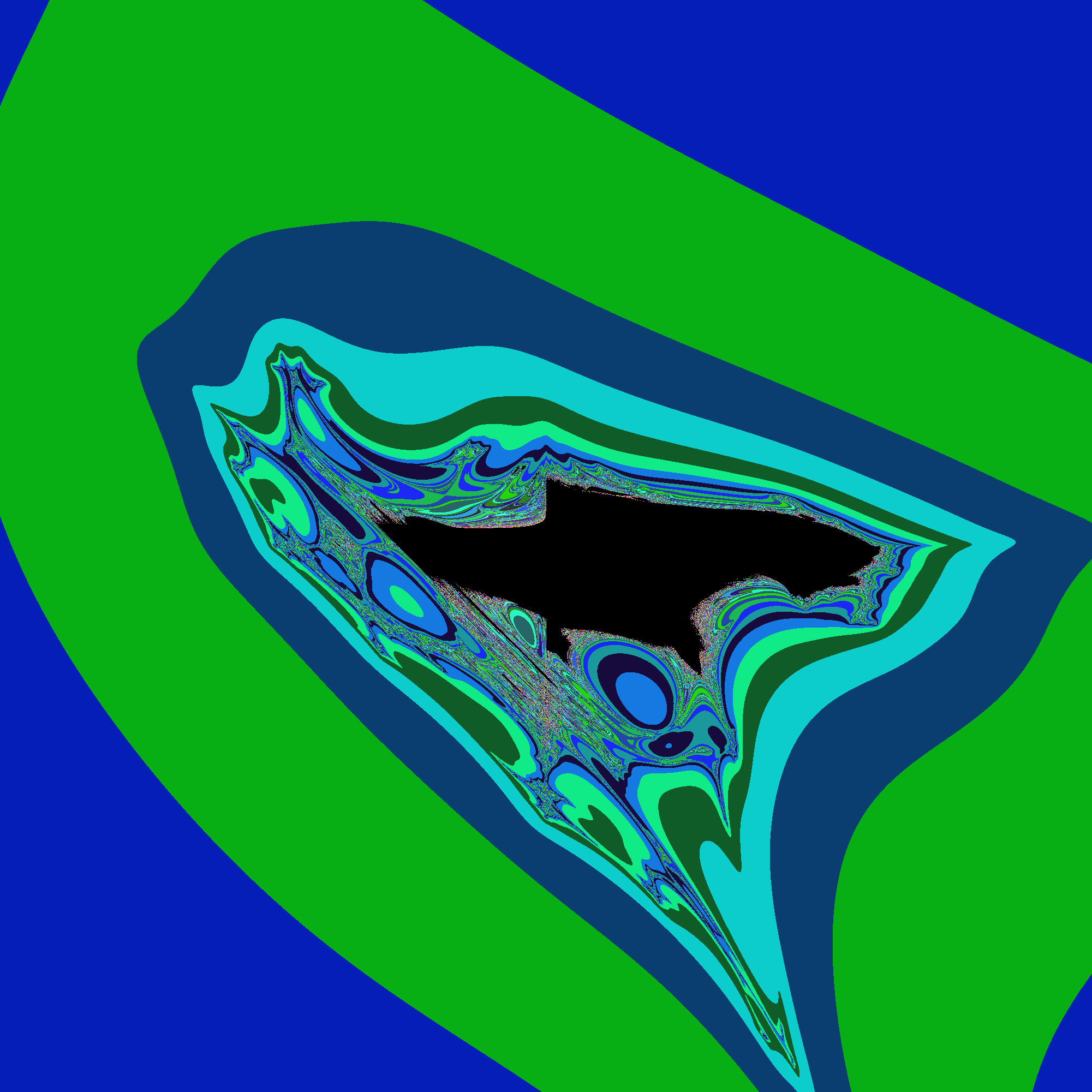} & \includegraphics[height=5cm]{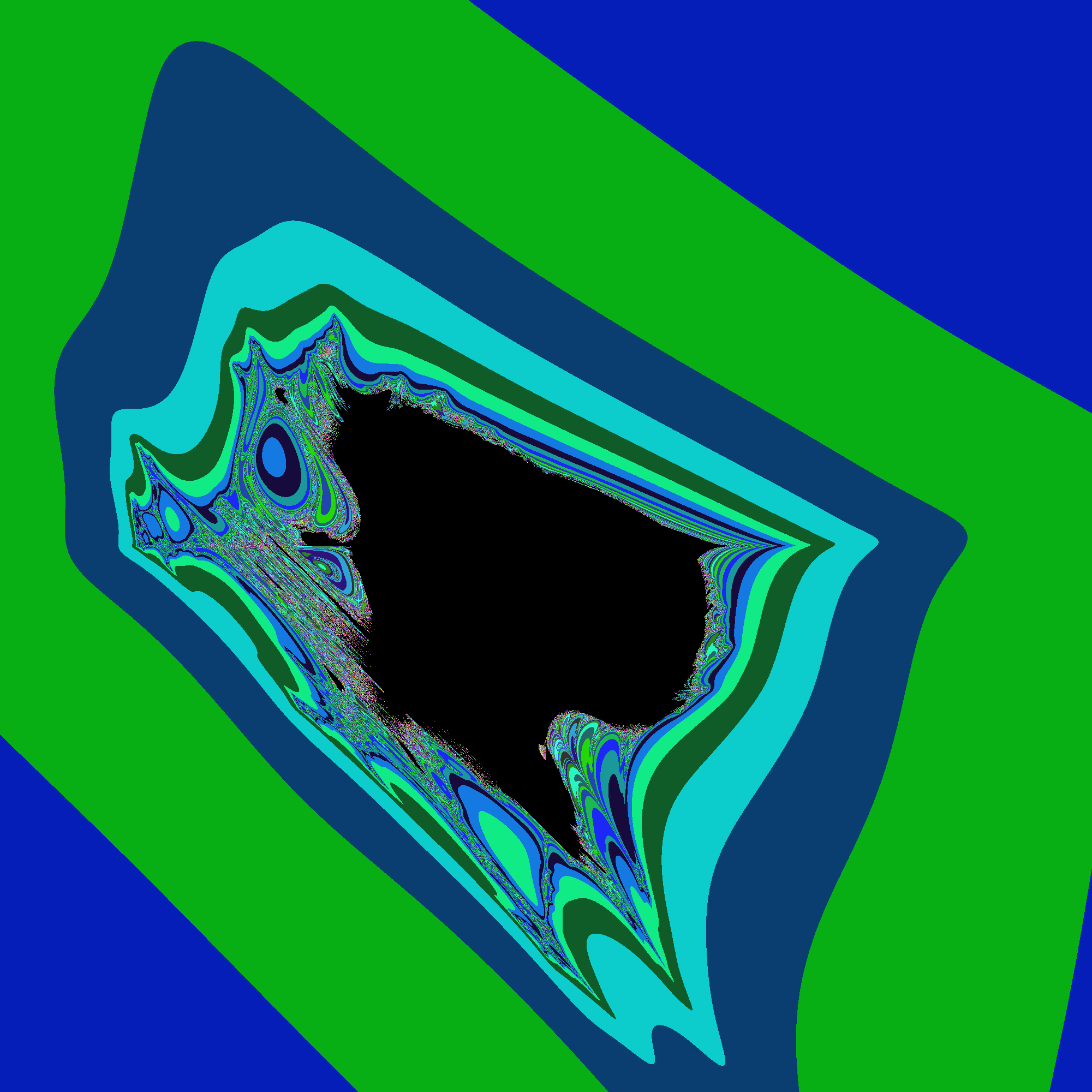} \\
			\includegraphics[height=5cm]{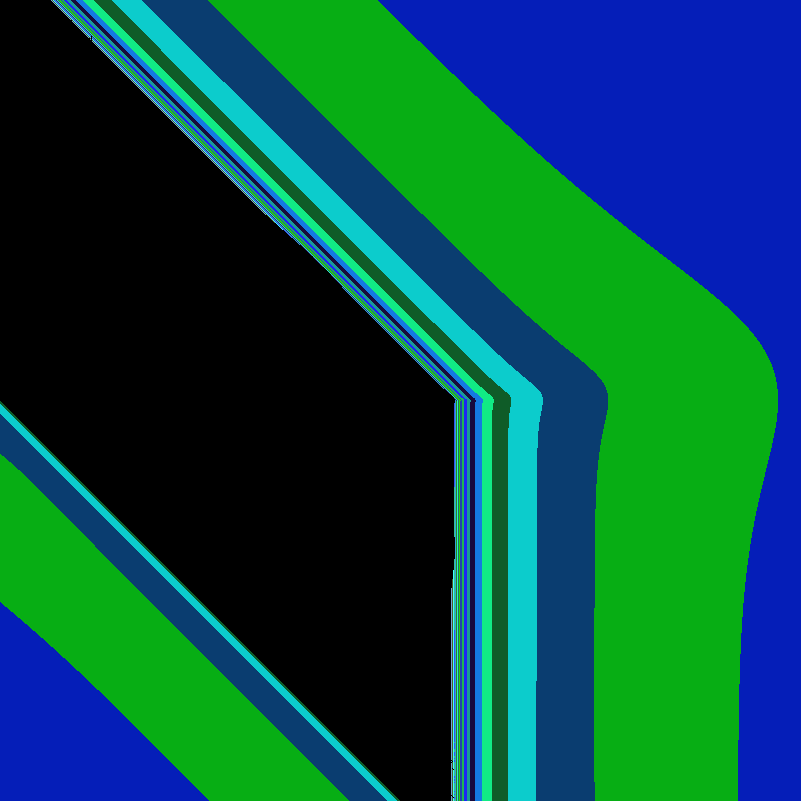} & \includegraphics[height=5cm]{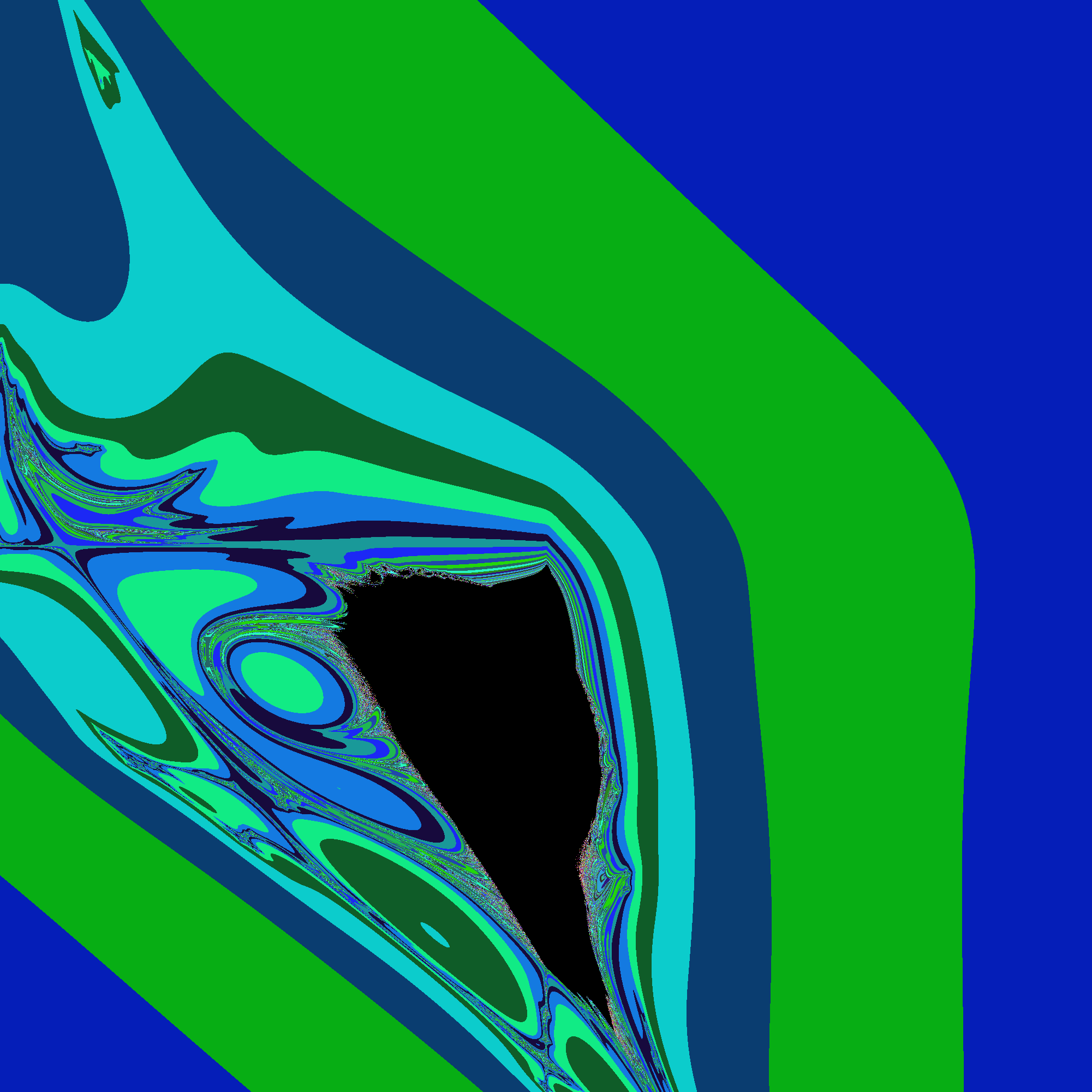} \\
			\includegraphics[height=5cm]{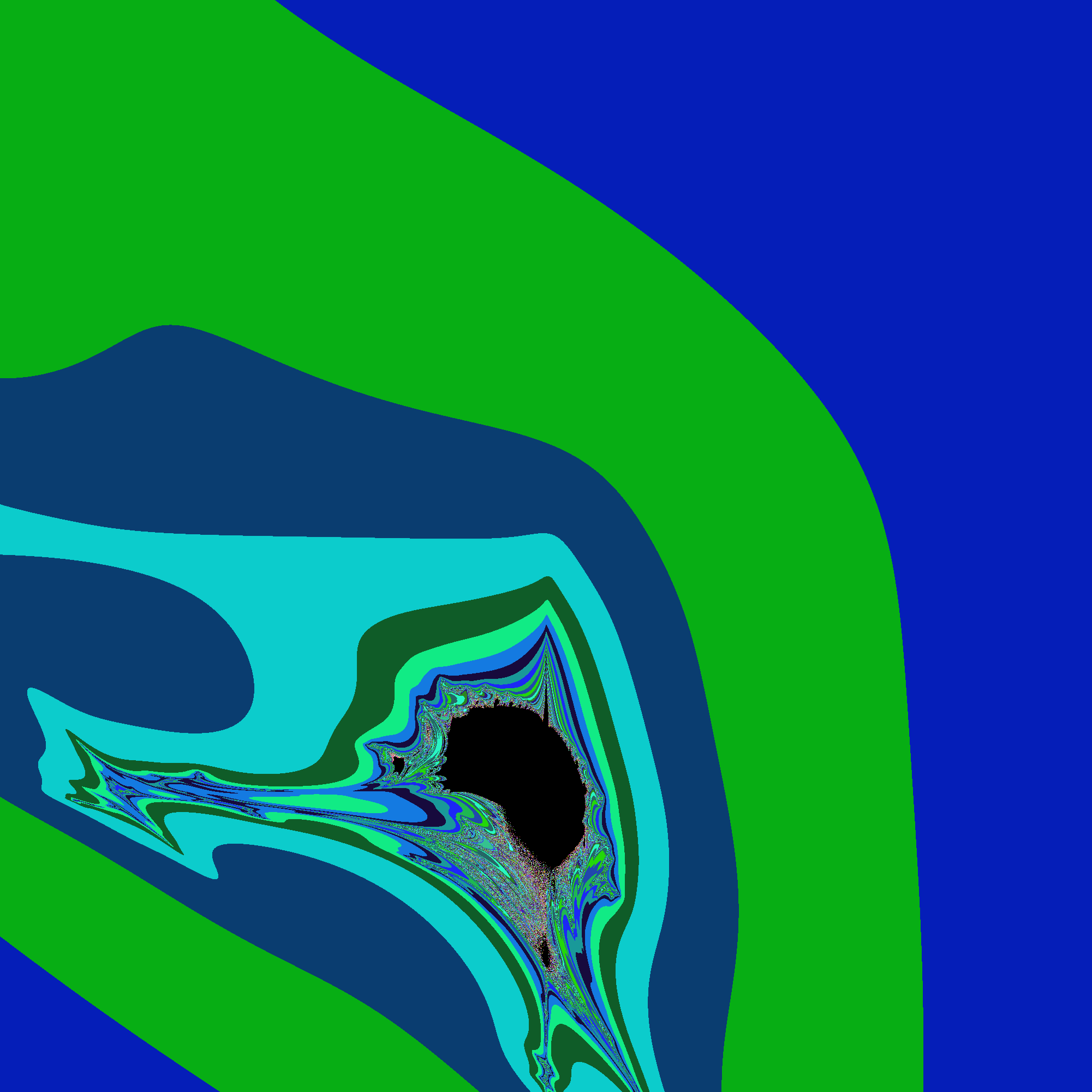} & \includegraphics[height=5cm]{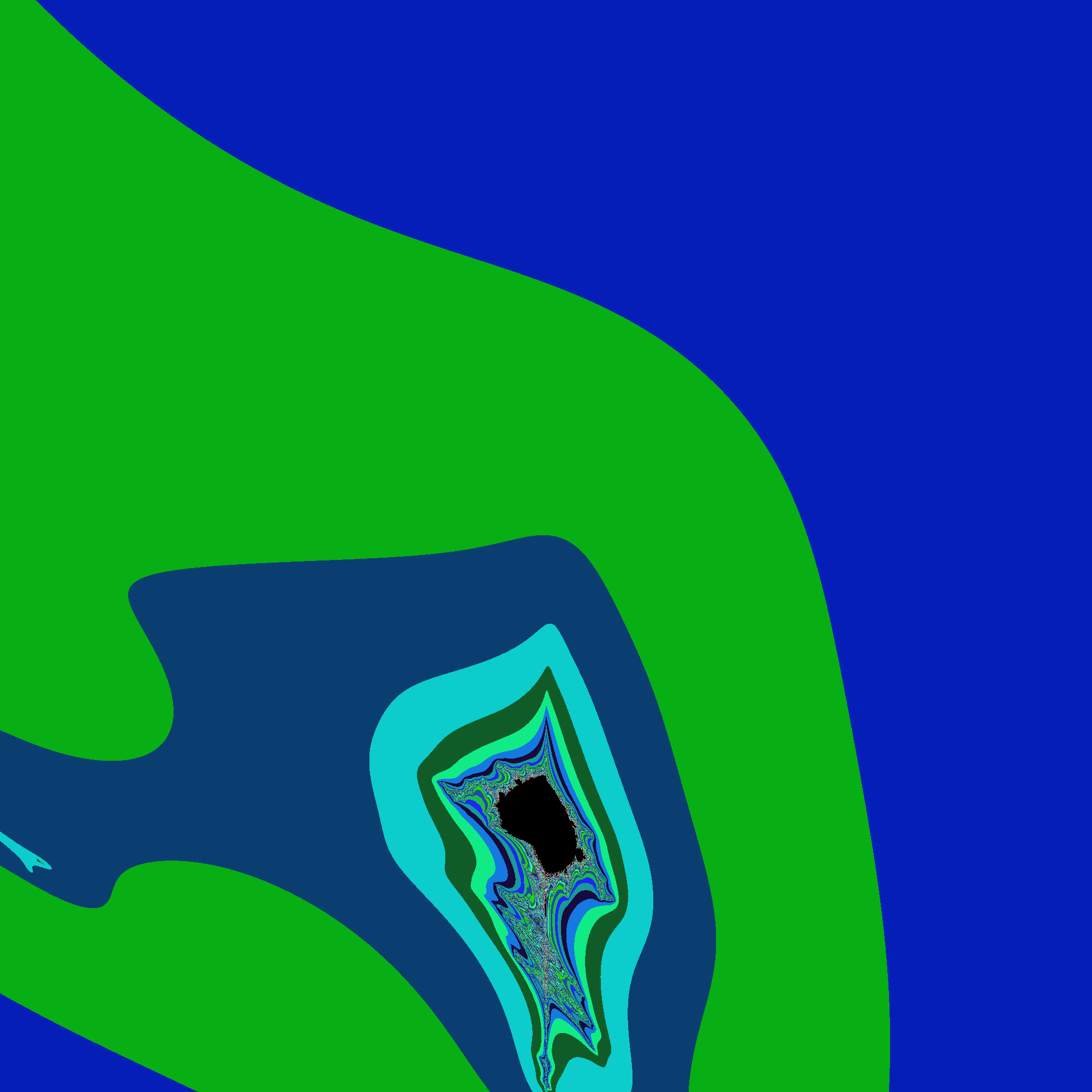}
		\end{tabular}
	\end{center}
	\caption{The rock-paper-scissors Mandelbrot set for \(z\) values \(-1\), \(\frac{-1}{2}\), \(0\), \(\frac{1}{3}\), \(\frac{2}{3}\) and \(1\)}
	\label{figure:RPS_Mandelbrot_slices}
\end{figure}

For convenience, a number of standard algebras are implemented by running the script \texttt{standard\char`_algebras.lua}. This script uses the general faculties from \texttt{algebras\char`_and\char`_magmas.lua} to instantiate as plain text files:
	\begin{itemize}
		\item the complex numbers, the quaternions, and the octonions,
		\item all groups of order between \(2\) and \(8\),
		\item all magmas of order \(2\), up to isomorphism, and
		\item the rock-paper-scissors magma,
	\end{itemize}
as well as the magma algebras for all the magmas listed above.

\section{Video game fractal visualization}
\label{section:video_game_fractal_visualization}
We can use the video game Luanti (\href{https://www.luanti.org}{https://www.luanti.org}) to visualize and explore the Mandelbrot set for various \(3\)-dimensional real algebras using the code found at \href{https://codeberg.org/caten/rpsland}{https://codeberg.org/caten/rpsland}. The world of this voxel-based game consists of a \(3\)-dimensional integer lattice of cubical nodes. Terrain is created by way of a rule which designates each node as air (i.e., empty), water, or stone. After this, the game engine may create more varied terrain from the stone, depending on the precise game and settings applied. We detail here the method by which terrain is constructed in the aforementioned code.

The script \texttt{common.lua} provides methods for performing addition and multiplication as well as computing lengths of vectors in a \(3\)-dimensional real algebra whose structure constants are supplied by the user. The default structure constants are those for the RPSions. This script also contains a function for assigning a natural number to a given point in the real algebra, depending on how quickly that point ``escapes''. Unlike in the plotting utility of the previous section, here we default to ``escaping to \(0\)'', which means that the program counts the number of iterations it takes for a point \(c\) to become within a certain small distance of \(0\) under iteration of \(z\mapsto z^2+c\). This tends to produce islands of terrain, as opposed to the usual escape to infinity from artistic fractal plotting, which produces mostly flat land interspersed with some fractal caves. The user can select ``escape to infinity'' at their option.

Since most points on the integer lattice might not give interesting results when iterated according to this paradigm, we make an adjustment. Given a point \((x,y,z)\), the user specifies scaling parameters so that the point ultimately passed into the fractal iteration function is of the form
	\[
		\left(\frac{x}{s+s_x}+x_0,\frac{y}{s+s_y}+y_0,\frac{z}{s+s_z}+z_0\right)
	\]
where \(s\) is a global scaling factor (the default is \(100\)), \(s_x\), \(s_y\), and \(s_z\) scale each coordinate individually, and \(x_0\), \(y_0\), and \(z_0\) are fixed offsets. In this way, the fine structure of the fractal may be rendered as terrain at an arbitrary level of detail.

After this, a separate fractal noise provided by the game engine may be added to the resulting triple for some additional variety. Also optional at this stage is the application of an ``embellishment'', such as passing all coordinates through the sine function. These embellishments make it possible to extend what would have been a single fractal island into a sea of fractal islands. Appropriate application of fractal noise can keep this world from becoming repetitive.

In any case, once the triple in question has been passed through the fractal iteration function the resulting natural number (essentially the number of steps it took for that point to escape) is linearly scaled to a value between the minimum and maximum height the user assigns for the terrain. If the original integer lattice point's height (the \(y\) coordinate, rather than the usual \(z\), in Luanti) is below this value, then the corresponding node is set as stone. Otherwise, it becomes air or water, depending on whether it is below sea level.

\section{Future directions}
\label{section:future_directions}
If the example of the RPSions is any indication, the notion of holomorphy introduced here may be interesting to explore in other finite dimensional real algebras. Among the magma algebras, two natural choices would be the group algebras \(\R[\mathbf{G}]\) for a finite group \(\mathbf{G}\) and the magma algebras \(\R[\mathbf{A}]\) for the RPS magmas from \cite{aten2020}. Once one has a handle on the holomorphic functions for a finite-dimensional real algebra \(\A\), a study of the attracting and repelling periodic points for familiar holomorphic maps may be made. We only just began this work for the Mandelbrot set analogue over \(\J\).

Rather than working with holomorphic functions, one might wish to pass to an analogue of meromorphic functions. If \(p(x)\) and \(q(x)\) are polynomial functions over \(\A\) (in the sense of universal algebra, where we cannot assume that variables associate), we can define corresponding ``rational functions''
	\[
		x\mapsto\lambda_{q(x)}^{-1}(p(x)) \text{ and } x\mapsto\rho_{q(x)}^{-1}(p(x))
	\]
for any values of \(x\) for which \(q(x)\) is left- or right-divisible, respectively. In the case where \(q(x)\) is always divisible on the appropriate side, except possibly when \(q(x)=0\), the resulting map might be viewed as a mapping on the \(n\)-dimensional analogue of the Riemann sphere. Just as we cannot assume that a given polynomial function is holomorphic over \(\A\), we cannot assume these analogues of rational functions are meromorphic, but this appears to be the appropriate starting point for that study.

Throughout this paper we have added the adjective ``finite-dimensional'' in many places where it might easily have been dropped. Certainly it can't be easily avoided everywhere, such as in \autoref{theorem:Cauchy-Riemann_equations}, where the upshot is that we obtain a finite system of PDEs. One might consider the generalization of our notion of holomorphy to arbitrary real algebras, perhaps demanding that the underlying vector space be a Hilbert space if the totally general case is too unwieldy.

The notion of holomorphy here is a special case of a more general family which may be of interest. Let \(P=\set{p_1,\dots,p_k}\) be a collection of \(k\) polynomials in two variables over \(\A\). We might say that a function \(f\colon\A\to\A\) is \emph{\(P\)-holomorphic} at a point \(x\in\A\) when
	\begin{enumerate}
		\item \(f\) is differentiable at \(x\) and
		\item there exists a unique \(\delta\in\A\) such that for all directions \(v\in\A\) and all \(1\le i,j\le k\) we have that
			\[
				p_i(v,\delta)=p_j(v,\delta)=\nabla_vf(x).
			\]
	\end{enumerate}
This unique value of \(\delta\) plays the role of \(f'(x)\) for this family of polynomials \(P\). The definition used in this paper corresponds to \(P=\set{xy,yx}\). Each choice of \(P\) will yield different families of holomorphic functions and would entail different relationships between algebraic properties of \(\A\) and rules of differentiation than those examined in \autoref{section:holomorphic_functions}.

On the subject of rules of differentiation, we could apply a similar analysis to that in \autoref{proposition:product_of_holomorphic_functions} for the Product Rule to any other rule we might formulate. For instance, we could ask for a characterization of those real algebras in which
	\[
		(fg)'(x)=f'(x)g'(x)
	\]
for \(f\) and \(g\) holomorphic.

While we have studied the Mandelbrot set for an algebra \(\A\), we haven't discussed the corresponding notion of Julia sets. In particular, our definition of the Mandelbrot set avoided any discussion of whether a point's membership in the Mandelbrot set says anything about the topology of the corresponding Julia set. It would be interesting to see whether this relationship generalizes from \(\C\) to other algebras.

\printbibliography

\end{document}